\documentclass[11pt]{article} 
\usepackage{amsfonts,amsmath,amssymb,mathabx}
\usepackage{fullpage}
\usepackage{color}
\usepackage{amsthm}
\usepackage{colortbl}
\usepackage{enumerate}
\usepackage{graphicx}
\usepackage{url}
\usepackage{subfigure}
\usepackage{xcolor}
\usepackage{dsfont}
\usepackage{fancybox}

\newtheorem{theorem}{Theorem}[section]

\newtheorem{definition}[theorem]{Definition}

\newtheorem{proposition}[theorem]{Proposition}
\newtheorem{remark}[theorem]{Remark}

\newcommand{\bone}{{\bf 1}}

\def\AA{\mathbb{A} }
\newcommand{\EE}{\mathbb{E} }
\newcommand{\FF}{\mathbb{F} }

\newcommand{\PP}{\mathbb{P} }

\newcommand{\RR}{\mathbb{R} }

\newcommand{\cF}{\mathcal{F} }

\newcommand{\cL}{\mathcal{L} }
\newcommand{\cP}{\mathcal{P}}

\newcommand{\cV}{\mathcal{V} }

\def\ch{\check}

\def\t{\tilde}
\def\o{\overline}

\title{ \textbf{The Master Equation for Large Population Equilibriums}\footnote{Paper presented at the  conference "Stochastic Analysis", University of Oxford, September 23, 2013}
}
\author{Ren\'e Carmona${}^{(a),}$\footnote{rcarmona@princeton.edu, Partially supported  by NSF: DMS-0806591 } \hspace*{.05cm},
Fran{\c c}ois Delarue${}^{(b),}$\footnote{delarue@unice.fr} \hspace*{.05cm} 
 \\ \\
(a) ORFE, Bendheim Center for Finance, Princeton University,
\\
Princeton, NJ  08544, USA.
\\
\\
(b) Laboratoire J.A. Dieudonn\'e, Universit\'e de Nice Sophia-Antipolis, 
\\
Parc Valrose, 06108 Nice Cedex 02, France.}

\date{}
\begin{document}
\maketitle

\begin{abstract}
We use a simple $N$-player stochastic game with idiosyncratic and common noises to introduce the concept of Master Equation originally proposed by Lions in his lectures at the \textit{Coll\`ege de France}. Controlling the limit $N\to\infty$ of the explicit solution of the $N$-player game, we highlight the stochastic nature of 
the limit distributions of the states of the players due to the fact that the random environment does not average out in the limit, and we recast the Mean Field Game (MFG) paradigm in a  set of coupled Stochastic Partial Differential Equations (SPDEs). The first one is a forward stochastic Kolmogorov equation giving the evolution of the conditional distributions of the states of the players given the common noise. The second is a form  of stochastic Hamilton Jacobi Bellman (HJB) equation providing the solution of the optimization problem when the flow of conditional distributions is given. Being 
highly coupled, the system reads as an infinite dimensional 
Forward Backward Stochastic Differential Equation (FBSDE). Uniqueness of a solution and its Markov property lead to the representation of the solution of the backward equation (i.e. the value function of the stochastic HJB equation) as a deterministic function of the solution of the forward Kolmogorov equation,  function which is usually called the \textit{decoupling field} of the FBSDE. 
The (infinite dimensional) PDE satisfied by this decoupling field is identified with the \textit{master equation}. We also show that this equation can be derived for other large populations equilibriums like those given by the optimal control of McKean-Vlasov stochastic differential equations.

The paper is written more in the style of a review than a technical paper, and we spend more time and energy motivating and explaining the probabilistic interpretation of the
Master Equation, than identifying the most general set of assumptions under which our claims are true.
\end{abstract}

\section{Introduction}
\label{se:introduction}

In several lectures given at the \textit{Coll\`ege de France},  P.L. Lions describes mean-field games by a single equation referred to as the \emph{master equation}. 
Roughly speaking, this equation encapsulates all the information about the Mean Field Game (MFG) into a single equation. The purpose of this paper is to review its theoretical underpinnings and to derive it for general MFGs with common noise.

The master equation is a Partial Differential Equation (PDE) in time, the state controlled by the players (typically an element of a Euclidean space, say $\RR^d$), and the probability distribution of this state. While the usual differential calculus is used in the time domain $[0,T]$ and in the state space $\RR^d$, the space $\cP(\RR^d)$ of probability measures needs to be endowed with a special differential calculus described in Lions' lectures, and explained in the notes Cardaliaguet wrote from these lectures, \cite{Cardaliaguet}. 
See also \cite{CarmonaDelarue_ap} and the appendix at the end of the paper.

Our goal is to emphasize the probabilistic nature of the master equation, as the associated characteristics are (possibly random) paths with values 
in the space $\RR^d \times {\mathcal P}(\RR^d)$. Our approach is especially enlightening for mean field games in a random environment (see Section \ref{se:regames} for definitions and examples), the simplest instances occurring in the presence of random shocks common to all the players. In that framework, the characteristics are given by the dynamics of 
$((X_{t},{\mathcal L}(X_{t}\vert W^0)))_{0 \leq t \leq T}$, where $(X_{t})_{0 \leq t \leq T}$ is the equilibrium trajectory of 
the game, as 
identified by the solution of the mean field game problem, and $({\mathcal L} (X_{t}\vert W^0))_{0\leq t \leq T}$ which denotes its conditional marginal distributions given the value of the common noise, 
describes the conditional distribution of the population 
at equilibrium. Examples of mean field games with a common noise were considered in \cite{GueantLasryLions.pplnm}, \cite{GomesSaude} and \cite{CarmonaFouqueSun}. Their theory is developed in the forthcoming paper \cite{CarmonaDelarueLacker} in a rather general setting.

As in the analysis of standard MFG models, the main challenge is the solution of a coupled system of a forward and a backward PDEs. However, in the random environment case, both equations are in fact stochastic PDEs (SPDEs). The forward SPDE is a Kolmogorov equation describing the dynamics of the conditional laws of the state given the common noise, and the backward SPDE is a stochastic Hamilton-Jacobi-Bellman equation describing the dynamics of the value function. Our contention is that this couple of SPDEs should be viewed as a Forward Backward Stochastic Differential Equation (FBSDE) in infinite dimension. For with this point of view, if some form of Markov property holds, it is natural to expect that the backward component can be written as a function of the forward component, this function being called the \textit{decoupling field}. In finite dimension, a simple application of It\^o's formula shows that when the decoupling field is smooth, it must satisfy a PDE. We use an infinite dimensional version of this argument to derive the master equation.  The infinite dimension version of It\^o's formula needed for the differential calculus chosen for the space of measures is taken 
from another forthcoming paper \cite{ChassagneuxCrisanDelarue} and is adapted to the case of a random environment 
in the appendix.

While the MFG approach does not ask for the solution of stochastic equations of the McKean-Vlasov type at first, the required 
fixed point argument identifies the equilibrium trajectory of the game as the \emph{de facto} solution 
of such an equation. This suggests that the tools developed for solving MFG problems could be reused for solving
optimal control problems of McKean-Vlasov dynamics. In the previous paper \cite{CarmonaDelarue_ap}, we established a suitable version of the stochastic Pontryagin principle for the control of McKean-Vlasov SDEs and highlighted the differences with the 
version of the stochastic Pontryagin principle used to tackle MFG models. Here we show in a similar way that our derivation of the master equation can be used as well for this other type of large population equilibrium problem.

The paper is organized as follows. Mean field games in a random environment are presented in Section \ref{se:regames}.
The problem is formulated in terms of a stochastic forward-backward system in infinite dimension.
A specific example, taken from \cite{CarmonaFouqueSun}, is exposed in Section
\ref{se:1stexample}. The master equation is derived explicitly. In Section 
\ref{se:ME}, we propose a more systematic approach approach of the master equation for large 
population control problems in a random environment. We consider 
both the MFG problem and the control of McKean-Vlasov dynamics. Another example, taken from
 \cite{GueantLasryLions.pplnm}, is revisited in Section \ref{se:2ndexample}. 
We end up with the proof of the chain rule along flow of random measures in the Appendix.

\vskip 4pt
\emph{When analyzed within the probabilistic framework of the stochastic maximum principle, MFGs with a common noise lead to the analysis of stochastic differential equations conditioned on the knowledge of some of the driving Brownian motions. These forms of conditioned forward stochastic dynamics are best understood in the framework of Terry Lyons' theory of rough paths. Indeed integrals and differentials with respect to the conditioned paths can be interpreted in the sense of rough paths while the meaning of the others can remain in the classical It\^o calculus framework. We thought this final remark was appropriate given the raison d'\^etre of the present volume, and our strong desire to convey our deepest appreciation to the man, and pay homage to the mathematician as a remarkably creative scientist. }

\section{Mean Field Games in a Random Environment}
\label{se:regames}
The basic purpose of mean-field game theory is to analyze asymptotic Nash equilibriums for large populations of individuals with mean-field interactions. This goes back to  
the earlier and simultaneous and independent works of Lasry and Lions in \cite{MFG1, MFG2, MFG3} and Caines, Huang and Malham\'e in \cite{HuangCainesMalhame2}.

Throughout the paper, we shall consider the problem when the individuals (also referred to as \emph{particles} or \emph{players}) are subject to two sources of noise: an idiosyncratic noise, independent from one particle to another, and a common one, accounting for the common environment in which the individuals evolve. We decide to 
model the environment by means of a zero-mean Gaussian white noise field $W^0=(W^0(\Lambda,B))_{\Lambda,B}$,
parameterized by the Borel subsets $\Lambda$ of a Polish space $\Xi$ and 
the Borel subsets $B$ of $[0,\infty)$, such that
$$
\EE\bigl[W^0(\Lambda,B)W^0(\Lambda',B')\bigr]=\nu\bigl(\Lambda \cap \Lambda'\bigr)|B\cap B'|,
$$
where we used the notation $|B|$ for the Lebesgue measure of a Borel subset of $[0,\infty)$. Here $\nu$ is a non-negative measure on $\Xi$, called the spatial intensity of $W^0$. Often we shall use the notation $W^0_t$ for $W^0(\,\cdot\,,[0,t])$, and most often, we shall simply take $\Xi=\RR^\ell$. 

We now assume that the dynamics in $\RR^d$, with $d \geq 1$, of the private state of player $i\in\{1,\cdots,N\}$ are given by stochastic differential equations (SDEs) of the form:
\begin{equation}
\label{eq:general model}
dX^i_t=b\bigl(t,X^i_t,\o\mu^N_t,\alpha^i_t\bigr)dt + \sigma\bigl(t,X^i_t,\o\mu^N_t,\alpha^i_t\bigr)dW^i_t + \int_{\Xi}\sigma^0\bigl(t,X^i_t,\o\mu^N_t,\alpha^i_t,\xi\bigr)W^0(d\xi,dt),
\end{equation}
where $W^1,\dots,W^N$ are $N$ independent Brownian motions, independent of $W^0$, 
all of them being defined on some filtered probability space $(\Omega,{\mathbb F}=({\mathcal F}_{t})_{t \geq 0},\PP)$.
For simplicity, we assume that $W^0,W^1,\dots,W^N$ are $1$-dimensional (multidimensional analogs can be handled along the same lines). The term $\o\mu^N_{t}$ denotes the empirical distribution of the particles are time $t$:
\begin{equation*}
\o \mu^N_{t} = \frac1N \sum_{i=1}^N \delta_{X_{t}^i}.
\end{equation*}
The processes $((\alpha_{t}^i)_{t \geq 0})_{1 \leq i \leq N}$ are progressively-measurable processes, with values in an open subset $A$ of some Euclidean space. They stand for control processes. The coefficients $b$, $\sigma$ and $\sigma^0$ are defined accordingly on 
$[0,T] \times \RR^d \times {\mathcal P}(\RR^d) \times A (\times \Xi)$
with values in $\RR^d$, in a measurable way, the set 
${\mathcal P}(\RR^d)$ denoting the space of probability measures on $\RR^d$ endowed with the topology of weak convergence.

The simplest example of random environment corresponds to a coefficient $\sigma^0$ independent of $\xi$. In this case, the random measure $W^0$ may as well be independent of the spatial component. In other words, we can assume that $W^0(d\xi,dt)=W^0(dt)=dW^0_t$, for an extra Wiener process $W^0$ independent of the space location $\xi$ and of the idiosyncratic noise terms $(W^i)_{1 \leq i \leq N}$, representing an extra source of noise which is \emph{common} to all the players. 

If we think of $W^0(d\xi,dt)$ as a random noise which is white in time (to provide the time derivative of a Brownian motion) and colored in space (the spectrum of the color being given by the Fourier transform of $\nu$), then the motivating  example  we should keep in mind is a function $\sigma^0$ of the form $\sigma^0(t,x,\mu,\alpha,\xi)\sim \sigma^0(t,x,\mu,\alpha)\delta(x-\xi)$ (with $\Xi = \RR^d$ and where $\delta$ is a mollified version of the delta function which we treat as the actual point mass at $0$ for the purpose of this informal discussion). In which case
the integration with respect to the spatial part of the random measure $W^0$ gives
$$
\int_{\RR^d}\sigma^0(t,X^i_t,\o\mu^N_t,\alpha^i_t,\xi)W^0(d\xi,dt)=\sigma^0(t,X^i_t,\o\mu^N_t)W^0(X^i_t,dt),
$$
which says that, at time $t$, the private state of player $i$ is subject to several sources of random shocks: its own idiosyncratic noise $W^i_t$, but also, an independent white noise shock picked up at the very location/value of his own private state.

\subsection{Asymptotics of the Empirical Distribution $\o\mu^N_t$}
\label{sub:spde}
The rationale for the MFG approach to the search for approximate Nash equilibriums for large games is based on several limiting arguments, including the analysis of the asymptotic behavior as $N\to\infty$ of the empirical distribution $\o\mu^N_t$ coupling the states dynamics of the individual players. By the symmetry of our model and de Finetti's law of large numbers, this limit should exist if we allow only exchangeable strategy profiles $(\alpha^1_t,\cdots,\alpha^N_t)$. This will be the case if we restrict ourselves to distributed strategy profiles of the form $\alpha^j_t=\alpha(t,X^j_t,\o\mu^N_t)$ for some deterministic (smooth) function $(t,x,\mu)\mapsto \alpha(t,x,\mu)\in A$.

In order to understand this limit, we can use an argument from 
propagation of chaos theory, as exposed in the lecture notes by Sznitman \cite{Sznitman}.
Another (though equivalent) way consists in discussing the action of $\bar{\mu}^N_{t}$
on test functions for $t \in [0,T]$, $T$ denoting some time horizon. Fixing a smooth test function $\phi$ with compact support in $[0,T]\times \RR^d$ and using It\^o's formula, we compute:
\begin{eqnarray*}
&&d\langle\phi(t,\,\cdot\,),\frac1N\sum_{j=1}^N\delta_{X^j_t}\rangle=\frac1N\sum_{j=1}^Nd\phi(t,X^j_t)\\
&&\phantom{??}=\frac1N\sum_{j=1}^N \bigg(\partial_t\phi(t,X^j_t) dt+\nabla\phi(t,X^j_t) \cdot dX^j_t+\frac12\text{trace}\{\nabla^2\phi(t,X^j_t)d[X^j,X^j]_t\}\bigg)\\
&&\phantom{??}=\frac1N\sum_{j=1}^N \partial_t\phi(t,X^j_t) dt +\frac1N\sum_{j=1}^N\nabla\phi(t,X^j_t) \cdot  \sigma \bigl(t,X^j_t,\o\mu^N_t,\alpha(t,X^j_t,\o\mu^N_t) \bigr)dW^j_t \\
&&\phantom{???}+ \frac1N\sum_{j=1}^N\nabla\phi(t,X^j_t)  \cdot  b\bigl(t,X^j_t,\o\mu^N_t,\alpha(t,X^j_t,\o\mu^N_t)\bigr)dt\\
&&\phantom{???}+\frac1N\sum_{j=1}^N\nabla\phi(t,X^j_t)  \cdot  \int_\Xi \sigma^0\bigl(t,X^j_t,\o\mu^N_t,\alpha(t,X^j_t,\o\mu^N_t),\xi\bigr)W^0(d\xi,dt)\\
&&\phantom{???}+\frac1{2N}\sum_{j=1}^N\text{trace}\bigg\{\bigg([\sigma\sigma^\dagger]
\bigl(t,X^j_t,\o\mu^N_t,\alpha(t,X^j_t,\o\mu^N_t)\bigr)
\\
&&\phantom{??????}+ \int_\Xi [\sigma^0\sigma^{0\dagger}]\bigl(t,X^j_t,\o\mu^N_t,\alpha(t,X^j_t,\o\mu^N_t),\xi\bigr)\nu(d\xi)\bigg)\nabla^2\phi(t,X^j_t)\bigg\}dt
\end{eqnarray*}
Our goal is to take the limit as $N\to\infty$ in this expression. Using the definition of the measures $\o\mu^N_t$ we can rewrite the above equality as:
\begin{equation*}
\begin{split}
&\langle\phi(t,\,\cdot\,),\o\mu^N_t\rangle - \langle\phi(0,\,\cdot\,),\o\mu^N_0\rangle
\\
&= O(N^{-1/2})  + \int_{0}^t 
\bigl\langle \partial_t\phi(s,\,\cdot\,),\o\mu^N_s\rangle ds 
+
\int_{0}^t 
\bigl\langle \nabla\phi(s,\,\cdot\,) \cdot b\bigl(s,\,\cdot\,,\o\mu^N_s,\alpha(s,\,\cdot\,,\o\mu^N_s)\bigr),
\o\mu^N_s\rangle ds
\\
&\phantom{?} + \frac12 \int_{0}^t 
\biggl\langle \text{trace}\bigg\{ \bigg([\sigma\sigma^\dagger]
\bigl(s,\,\cdot\,,\o\mu^N_s,\alpha(s,\,\cdot\,,\o\mu^N_s) \bigr)
\\
&\hspace{100pt}+
\int_\Xi [\sigma^0\sigma^{0\dagger}]\bigl(s,\,\cdot\, ,\o\mu^N_s,\alpha(s,\,\cdot\,,\o\mu^N_s),\xi \bigr)\nu(d\xi)\bigg)\nabla^2\phi(t,\,\cdot\,)\bigg\},\o\mu^N_s \biggr\rangle ds
\\
&\phantom{?}+ \int_{0}^t 
\bigl\langle \nabla\phi(s,\,\cdot\,) \cdot \int_\Xi \sigma^0\bigl(s,\,\cdot\, ,\o\mu^N_s,\alpha(s,\,\cdot\,,\o\mu^N_s),\xi
\bigr)W^0(d\xi,ds),\o\mu^N_s \bigr\rangle, 
\end{split}
\end{equation*}
which shows (formally) after integration by parts that, in the limit $N\to\infty$, 
$$
\mu_t=\lim_{N\to\infty}\o\mu^N_t
$$
appears as a solution of the Stochastic Partial Differential Equation (SPDE) 
\begin{equation}
\label{fo:spde}
\begin{split}
&\displaystyle d\mu_t=- \nabla \cdot \bigl[b \bigl(t,\,\cdot\,,\mu_t,\alpha(t,\,\cdot\,,\mu_t) \bigr)\mu_t \bigr] dt
-\nabla \cdot \bigg(\int_\Xi \sigma^0 \bigl(t,\,\cdot\, ,\mu_t,\alpha(t,\,\cdot\,,\mu_t),\xi \bigr)W^0(d\xi,dt)\mu_t\bigg) 
\\
&\displaystyle + \frac12 \text{trace}\bigg[\nabla^2\bigg( \bigl[\sigma\sigma^\dagger \bigr]
\bigl(t,\,\cdot\,,\mu_t,\alpha(t,\,\cdot\,,\mu_t)\bigr)+
\int_\Xi \bigl[\sigma^0\sigma^{0\dagger}\bigr]\bigl(t,\,\cdot\, ,\mu_t,\alpha(t,\,\cdot\,,\mu_t),\xi\bigr)\nu(d\xi)\bigg)\mu_t\bigg] dt.
\end{split}
\end{equation}
This SPDE reads as a stochastic Kolmogorov equation. It describes the flow of marginal distributions of the solution of a conditional McKean-Vlasov equation, namely:
\begin{equation}
\label{eq:MKV:SDE}
\begin{split}
dX_t&=b\bigl(t,X_t,\mu_{t},\alpha(t,X_{t},\mu_{t})\bigr)dt + 
\sigma\bigl(t,X_{t},\mu_{t},\alpha(t,X_{t},\mu_{t})\bigr)dW_{t} 
\\
&\hspace{15pt}+ 
\int_{\Xi}\sigma^0(t,X_{t},\mu_{t},\alpha(t,X_{t},\mu_{t}),\xi\bigr)W^0(d\xi,dt),
\end{split}
\end{equation}
subject to the constraint
$\mu_{t} = {\mathcal L}(X_{t} \vert {\mathcal F}^0_{t})$,
where $\FF^0=(\cF^0_t)_{t\ge 0}$ is the filtration generated by the spatial white noise measure $W^0$. Throughout the whole paper, the letter ${\mathcal L}$ refers to the law,
so that ${\mathcal L}(X_{t} \vert {\mathcal F}^0_{t})$ denotes the conditional law
of $X_{t}$ given ${\mathcal F}^0_{t}$. 
The connection between \eqref{fo:spde} and \eqref{eq:MKV:SDE}
can be checked by expanding $(\langle \phi(t,\cdot),\mu_{t} \rangle = {\mathbb E}(\phi(X_{t})\vert {\mathcal F}_{t}^0))_{0 \leq t \leq T}$ by means of It\^o's formula.

For the sake of illustration we rewrite this SPDE in a few particular cases which we will revisit later on:

\vskip 2pt
1. If we assume that $\sigma(t,x,\mu,\alpha)\equiv \sigma$ is a constant, that $\sigma^0(t,x,\mu,\alpha)\equiv\sigma^0(t,x)$ is also uncontrolled and that the spatial white noise is actually scalar, namely $W(d\xi,dt)=dW^0_t$ for a scalar Wiener process $W^0$ independent of the Wiener processes $(W^i)_{i\ge 1}$, then the stochastic differential equations giving the dynamics of the state of the system read
\begin{equation}
\label{fo:francoiscn}
dX^i_t=b(t,X^i_t,\o\mu^N_t,\alpha^i_t)dt + \sigma dW^i_t + \sigma^0(t,X^i_t)dW^0_t,\qquad i=1,\cdots,N
\end{equation}
and the limit $\mu_t$ of the empirical distributions satisfies the equation
\begin{equation}
\label{fo:francoismut}
\begin{split}
&d\mu_t= - \nabla \cdot \bigl[b \bigl(t,\,\cdot\,,\mu_t,\alpha(t,\,\cdot\,,\mu_t) \bigr)\mu_t \bigr] dt
-\nabla \cdot\big(\sigma^0(t,\,\cdot\, )dW^0_t\mu_t\big) 
\\
&\phantom{?????}+ \frac12 \text{trace}\Big[\nabla^2\Big(
\bigl[\sigma\sigma^\dagger + \sigma^0\sigma^{0\dagger} \bigr](t,\,\cdot\, )\Big)\mu_t\Big] dt. 
\end{split}
\end{equation}
Writing the corresponding version \eqref{eq:MKV:SDE}, rough paths theory would permit 
to express the dynamics of the path $(X_{t})_{t \geq 0}$ conditional on the values of $W^0$. This would be another way to express the dynamics of the conditional marginal 
laws of $(X_{t})_{t \geq 0}$ given $W^0$. 
\vskip 2pt

2. Note that, when the ambient noise is not present (i.e. either $\sigma^0\equiv 0$ or $W^0\equiv 0$), this SPDE reduces to a deterministic PDE. It is the Kolmogorov equation giving the forward dynamics of the distribution at time $t$ of the nonlinear diffusion process $(X_{t})_{t \geq 0}$ (nonlinear in McKean's sense).

\subsection{Solution Strategy for Mean Field Games}
\label{se:strategy}
When players are assigned a cost functional, a natural (and challenging) question is to determine 
equilibriums within the population. A typical framework is to assume that the cost to player $i$, for any $i \in \{1,\dots,N\}$, writes
\begin{equation*}
J^i(\alpha^1,\dots,\alpha^N) = \EE \biggl[ \int_{0}^T f\bigl(t,X_{t}^i,\o\mu^N_{t},\alpha_{t}^i\bigr) dt + g\bigl(X_{T}^i,\o{\mu}_{T}^N\bigr) \biggr],
\end{equation*}
for some functions $f: [0,T] \times \RR^d \times {\mathcal P}(\RR^d) \times A \rightarrow \RR$
and $g : \RR^d \times {\mathcal P}(\RR^d)\rightarrow \RR$. Keep in mind the fact that the cost 
$J^i$ depends on all the controls $((\alpha^j_{t})_{0 \leq t \leq T})_{j \in \{1,\dots,N\}}$
through the flow of empirical measures $(\o\mu^N_{t})_{0 \leq t \leq T}$. 

In the search for a Nash equilibrium $\alpha$, one assumes that all the players $j$ but one keep the same strategy profile $\alpha$, and the remaining player deviates from this strategy in the hope of being better off. If the number of players is large (think $N\to\infty$), one expects that the empirical measure $\o\mu^N_t$ will not be affected much by the deviation of one single player, and for all practical purposes, one can assume that the empirical measure $\o\mu^N_t$ is approximately equal to its limit $\mu_t$.
So in the case of large symmetric games, the search for approximate Nash equilibriums could be done through the solution of 
the optimization problem of one single player (typically the solution of a stochastic control problem instead of a large game)
when the empirical measure $\o\mu^N_t$ is replaced by the solution $\mu_t$ of the SPDE \eqref{fo:spde}  appearing
in this limiting regime, the `$\alpha$' plugged in \eqref{fo:spde} denoting the strategy used by the players at equilibrium. 

\vskip 2pt
The implementation of this method can be broken down into three steps for pedagogical reasons:
\begin{enumerate}\itemsep=-2pt
\item[(i)] Given an initial distribution $\mu_{0}$ on $\RR^d$, fix an arbitrary measure valued adapted stochastic process $(\mu_t)_{0\le t\le T}$ over the probability space of the random measure $W^0$. It stands for a possible candidate for being  a Nash equilibrium. 
\item[(ii)] Solve the (standard) stochastic control problem (with random coefficients)
\begin{equation}
\label{fo:mfgcontrolpb}
\inf_{(\alpha_{t})_{0 \leq t \leq T}}\EE\left[\int_0^Tf(t,X_t,\mu_t,\alpha_t)dt + g(X_T,\mu_T)\right]
\end{equation}
subject to  
$$
dX_t=b\bigl(t,X_t,\mu_t,\alpha_t\bigr)dt + \sigma\bigl(t,X_t,\mu_t,\alpha_t\bigr) dW_t+\int_\Xi\sigma^0\bigl(t,X_t,\mu_t,\alpha_t,\xi\bigr)W^0(d\xi,dt), 
$$
with $X_{0} \sim \mu_{0}$,
over controls in feedback form, Markovian in $X$ conditional on the past of the flow of random
measures $(\mu_{t})_{0 \leq t \leq T}$.
\item[(iii)] Plug the optimal feedback function $\alpha(t,x,\mu_t)$ in the SPDE
\eqref{fo:spde}. Then, determine the measure valued stochastic process $(\mu_t)_{0\le t\le T}$  so that the solution of the SPDE \eqref{fo:spde} be precisely $(\mu_{t})_{0 \leq t \leq T}$ itself. 
\end{enumerate}
Clearly, this last item requires the solution of a fixed point problem in an infinite dimensional space, while the second item involves the solution of an optimization
problem in a space of stochastic processes. 
Thanks to the connection between the SPDE \eqref{fo:spde} and the 
McKean-Vlasov equation \eqref{eq:MKV:SDE},
the fixed point item (iii) reduces to the search for a flow of random measures $(\mu_{t})_{0 \leq t \leq T}$ such that the law of the optimally controlled process (resulting from the solution of the second item) is in fact $\mu_t$, i.e.
$$
\forall t\in[0,T],\quad \mu_t = {\mathcal L}(X_{t} \vert {\mathcal F}_{t}^0). 
$$

In the absence of the ambient random field noise term $W^0$, the measure valued adapted stochastic process $(\mu_t)_{0\le t\le T}$ can be taken as a 
deterministic function $[0,T]\ni t\mapsto \mu_t\in \cP(\RR^d)$ and the control problem in item (ii) is a standard Markovian control problem. Moreover, the fixed point item (iii) reduces to the search for a deterministic flow of measures $[0,T]\ni t\mapsto \mu_t\in \cP(\RR^d)$ such that the optimally controlled process (resulting from the solution of the second item) 
satisfies ${\mathcal L}(X_t)=\mu_t$.

\subsection{Stochastic HJB Equation}
In this subsection, we study the stochastic control (ii) when the flow of random measures $\mu=(\mu_t)_{0\le t\le T}$ is fixed. 
Optimization is performed over the set $\AA$ of $\FF$-progressively measurable $A$-valued processes  $(\alpha_t)_{0\le t\le T}$ satisfying
$$
\EE\int_0^T|\alpha_t|^2dt<\infty.
$$
For each $(t,x)\in [0,T]\times\RR^d$, we let $(X^{t,x}_s)_{t\le s\le T}$ be the solution of the stochastic differential equation (being granted that it is well-posed)
\begin{equation}
\label{fo:cn_state}
dX_s=b(s,X_s,\mu_s,\alpha_s)ds+\sigma(s,X_s,\mu_s,\alpha_s) dW_s+\int_\Xi \sigma^0(s,X_s,\mu_a,\alpha_s,\xi)W^0(d\xi,ds),
\end{equation}
with $X_{t}=x$. 
With this notation, we define the (conditional) cost
\begin{equation}
\label{fo:stochastic_value}
J^\mu_{t,x}\bigl((\alpha_{s})_{t \leq s \leq T}\bigr)=\EE\bigg[\int_t^Tf(s,X^{t,x}_s,\mu_s,\alpha_s)ds+g(X^{t,x}_T,\mu_T)\Big| \cF^0_t\bigg]
\end{equation}
and the (conditional) value function 
\begin{equation}
\label{fo:stochastic_value:2}
V^\mu(t,x)=\underset{(\alpha_{s})_{t \leq s \leq T} \in\AA}{\text{ess inf}} \ J^\mu_{t,x}
\bigl((\alpha_{s})_{t \leq s \leq T}\bigr).
\end{equation}
We shall drop the superscript and write $X_s$ for $X^{t,x}_s$ when no confusion is possible. Under some regularity assumptions, we can show that, for each $x\in\RR^d$, $(V(t,x))_{0\le t\le T}$ is an $\FF^0$-semi-martingale and deduce, 
by identification of its It\^o decomposition, 
that it solves a form of stochastic Hamilton-Jacobi Bellman (HJB) equation.
Because of the special form of the state dynamics \eqref{fo:cn_state}, we introduce the
(nonlocal) operator symbol
\begin{equation}
\label{fo:cn_symbol}
\begin{split}
&L^*\bigl(t,x,y,z,(z^0(\xi))_{\xi \in \Xi} \bigr)
\\
&= \inf_{\alpha\in A}\bigg[ b(t,x,\mu_{t},\alpha)\cdot y
+
\frac12\text{trace}\bigl([\sigma\sigma^\dagger](t,x,\mu_{t},\alpha)
\cdot z \bigr) +f(t,x,\mu_{t},\alpha)
\\
&\hspace{15pt} + \frac12\text{trace}\biggl( \int_{\Xi}[\sigma^0\sigma^{0\dagger}](t,x,\mu_{t},\alpha,\xi) d \nu(\xi)]\cdot z\biggr)
+\int_{\Xi}
\sigma^0\bigl(t,x,\mu_{t},\alpha,\xi)\cdot z^0(\xi) d\nu(\xi) \bigg].
\end{split}
\end{equation}
Assuming that the value function is smooth enough, we can use a generalization of the dynamic programming principle to the present set-up of conditional value functions to show that $V^\mu(t,x)$ satisfies a form of stochastic HJB equation as given by a parametric family of BSDEs in the sense that:
\begin{equation}
\label{fo:sHJB}
\begin{split}
V^\mu(t,x)&= g(x)+\int_t^T L^*\bigl(s,x,\partial_{x}V(s,x),\partial^2_{x}V^\mu(s,x),
(Z^{\mu}(s,x,\xi))_{\xi \in \Xi} \bigr)ds 
\\
&\hspace{15pt}+\int_t^T Z^\mu(s,x,\xi)W^0(d\xi,ds).
\end{split}
\end{equation}
Noticing that $W^0$ enjoys the martingale representation theorem (see Chapter 1 in 
\cite{Nualart}), this result can be seen as part of the folklore of the theory of backward SPDEs 
(see for example \cite{Peng_sHJB} or \cite{Ma_sHJB}).

\subsection{Towards the Master Equation}
The definition of $L^*$ in \eqref{fo:cn_symbol}
suggests that 
the optimal feedback in \eqref{fo:stochastic_value} could be identified as a
function $\hat{\alpha}$ of $t$, $x$, $\mu_{t}$, $V^{\mu}(t,\cdot)$ and 
$Z^{\mu}(t,\cdot,\cdot)$ realizing the infimum appearing in the definition of $L^*$. Plugging such a choice for $\alpha$ in the SPDE
\eqref{fo:spde}, we deduce that the fixed point condition in the item (iii) of a definition of an
MFG  equilibrium  could be reformulated in terms of an infinite dimensional FBSDE, the forward component of which being the Kolmogorov SPDE \eqref{fo:spde} (with the specific choice of  $\alpha$)
and the backward component the stochastic HJB equation \eqref{fo:sHJB}. The forward variable would be $(\mu_{t})_{0\leq t \leq T}$ and the backward one would be 
$(V^{\mu}(t,\cdot))_{0 \leq t \leq T}$. Standard FBSDE theory suggests the existence of a \textit{decoupling field} expressing the backward variable in terms of the forward one, in other words that $V^{\mu}(t,x)$ could be written as $V(t,x,\mu_{t})$ for some function $V$, or equivalently, that
$V^{\mu}(t,\cdot)$ could be written as $V(t,\cdot,\mu_{t})$. Using a special form of It\^o's change of variable formula proven in the appendix at the end of the paper, these
decoupling fields are easily shown, at least when they are smooth, to satisfy PDEs or SPDEs in the case of FBSDEs with random coefficients. The definition of the special notion of smoothness required for this form of It\^o formula is recalled in the appendix. This is our hook to Lions's master equation. In order to make this point transparent, we strive in the sequel, to provide a better understanding of the mapping $V : [0,T] \times \RR^d \times {\mathcal P}(\RR^d)
\rightarrow \RR$ and of its dynamics.

\section{An Explicitly Solvable Model}
\label{se:1stexample}
This section is devoted to the analysis of an explicitly solvable model. It was introduced and solved in \cite{CarmonaFouqueSun}. We reproduce the part of the solution which is relevant to the present discussion.  Our interest in this model is the fact that the finite player game can be solved explicitly and the limit $N\to\infty$ of the solution can be controlled. We shall use it as a motivation and testbed for the introduction of the master equation of mean field games with a common noise.

\subsection{Constructions of Exact Nash Equilibria for the $N$-Player Game}
We denote by $X^i_t$ the log-capitalization of a bank $i\in\{1,\cdots,N\}$ at time $t$. We assume that  each bank controls its rate of borrowing and lending 
through the drift of $X^i_t$ in such a way that:
\begin{equation}
\label{fo:Xit}
dX^i_t=\left[a(m_t^N-X^i_t)+\alpha^i_t\right]dt +\sigma \bigg(\sqrt{1-\rho^2} dW^i_t+\rho dW^0_t\bigg),
\end{equation}
where $W^i_t, i=0,1,\dots, N$  are independent  scalar Wiener processes, $\sigma>0$, $a\geq 0$, and $m^N_t$ denotes the sample mean of the $X^i_t$ as defined by $m_t^N=(X^1_t+\cdots + X^N_t)/N$. So, in the notation introduced in \eqref{eq:general model}, we have
$$
b(t,x,\mu,\alpha) = a(m-x)+\alpha,\qquad\text{with}\quad m=\int_{\RR} x\mu(dx),
$$
since the drift of $(X^i_t)_{t \geq 0}$ at time $t$ depends only upon $X^i_t$ itself and the mean $m^N_t$ of the empirical distribution $\o\mu_t^N$ of $X_t=(X^1_t,\cdots,X^N_t)$, and
$$
\sigma(t,x,\mu,\alpha)=\sigma\sqrt{1-\rho^2},\qquad\text{and}\qquad \sigma^0(t,x)=\sigma\rho.
$$
Bank $i\in\{1,\cdots,N\}$ controls its rate of lending and borrowing (to a central bank) at time $t$ by choosing the control $\alpha^i_t$ in order to minimize
\begin{equation}\label{objectives}
J^i(\alpha^1,\cdots,\alpha^N)=\EE\bigg[\int_0^T f(t,X^i_t,\o\mu^N_t,\alpha^i_t)dt+g(X_{T}^{i},\o\mu^N_T)\bigg],
\end{equation}
where the running and terminal cost functions $f$ and $g$ are given by:
\begin{equation}
\label{fi}
\begin{split}
&f(t,x,\mu,\alpha) = \frac{1}{2}\alpha^2-q\alpha(m-x)+\frac{\epsilon}{2}(m-x)^2,
\\
&g(x,\mu) = \frac{c}{2}(m-x)^2,
\end{split}
\end{equation}
where, as before, $m$ denotes the mean of the measure $\mu$.
Clearly, this is a {\it Linear-Quadratic} (LQ) model and, thus, its solvability should be equivalent to the well-posedness of a matricial Riccati equation. However, given the special
structure of the interaction, the Ricatti equation is in fact scalar and can be solved explicitly as we are about to demonstrate.

Given an $N$-tuple $(\hat{\alpha}^i)_{1 \leq i \leq N}$ 
of functions from $[0,T] \times \RR$ into $\RR$, we define, for each $i\in\{1,\cdots,N\}$, 
the related value function $V^i$ by:
\[
V^i(t,x^1,\dots,x^N)=\inf_{(\alpha^i_{s})_{t \leq s \leq T}}
\EE\bigg[\int_t^T f\bigl(s,X^i_s,\mu^N_s,\alpha^i_s\bigr)ds+g_i(X_{T}^{i},\o\mu^N_T)\Big| X_t=x\bigg],
\]
with  the cost functions $f$ and $g$ given in (\ref{fi}), and where the dynamics of $(X_s^1,\dots,X_{s}^N)_{t \leq s \leq T}$ are given in \eqref{fo:Xit}
with $X_{t}^j = x^j$ for $j\in \{1,\dots,N\}$ and $\alpha^j_{s} = \hat{\alpha}^j(s,X_{s}^j)$ for $j \not = i$. 
By dynamic  programming, the $N$ scalar functions $V^i$ must satisfy the system of HJB equations:
\begin{equation*}
\begin{split}
&\partial_{t} V^i(t,x) + \inf_{\alpha \in \RR}
\bigl\{ \bigl( a(\o{x}- x^{i}) + \alpha \bigr) \partial_{x^{i}} V^i(t,x) + \frac12 \alpha^2 
- q \alpha \bigl( \o{x} - x^{i} \bigr) \bigr\}+ \frac{\epsilon}2 (\bar{x} - x^{i})^2
\\
&\hspace{15pt} + \sum_{j \not = i} \bigl( a(\o{x} - x^j) + \hat{\alpha}^j(t,x^j) \bigr) \partial_{x^{j}} V^j(t,x) +  \frac{\sigma^2}{2}\sum_{j=1}^N\sum_{k=1}^N\left(\rho^2+\delta_{j,k}(1-\rho^2)\right)\partial^2_{x^jx^k}V^i(t,x) = 0, 
\end{split}
\end{equation*}
for $(t,x) \in [0,T] \times \RR^N$, where 
we use the notation $\overline{x}$ for the mean $\overline{x}=(x^1+\cdots+x^N)/N$ and
with the terminal condition $V^i(T,x) = (c/2)(\o{x}-x^{i})^2$. The infima in these HJB equations 
can be computed explicitly:
\begin{equation*}
\begin{split}
& \inf_{\alpha \in \RR}
\bigl\{ \bigl( a(\o{x}- x^{i}) + \alpha \bigr) \partial_{x^{i}} V^i(t,x) + \frac12 \alpha^2 
- q \alpha \bigl( \o{x} - x^{i} \bigr) \bigr\}
\\
&\hspace{15pt} = a(\o{x}- x^{i})  \partial_{x^{i}} V^i(t,x) - \frac12 \bigl[ q \bigl( \o{x} - x^{i} \bigr)  - \partial_{x^{i}} V^i(t,x) \bigr]^2,  
\end{split}
\end{equation*}
the infima being attained for 
\begin{equation*}
\alpha =  q \bigl( \o{x} - x^{i} \bigr)  - \partial_{x^{i}} V^i(t,x).
\end{equation*}
Therefore, the Markovian strategies $(\hat{\alpha}^i)_{1 \leq i \leq N}$
forms a Nash equilibrium if $\hat{\alpha}^i(t,x) = q \bigl( \o{x} - x^{i} \bigr)  - \partial_{x{i}} V^i(t,x)$, which suggests to solve the system of $N$ coupled HJB equations:
\begin{equation}
\label{fo:HJB}
\begin{split}
&\partial_{t}V^i+
\sum_{j=1}^N\left[(a+q)\left(\overline{x}-x^j\right)-\partial_{x^j}V^j\right]\partial_{x^j}V^i+\frac{\sigma^2}{2}\sum_{j=1}^N\sum_{k=1}^N\left(\rho^2+\delta_{j,k}(1-\rho^2)\right)\partial_{x^jx^k}^2V^i
\\
&\phantom{?????????????}+\frac{1}{2}(\epsilon-q^2)\left(\overline{x}-x^i\right)^2+\frac{1}{2}(\partial_{x^i}V^i)^2=0,\qquad i=1,\cdots,N,
\end{split}
\end{equation}
with the same boundary terminal condition as above. 
Then, the feedback functions $\hat\alpha^i(t,x)=q(\o x-x^i)-\partial_{x^i}V^i(t,x)$ are expected to give the optimal Markovian strategies. Generally speaking, these systems of HJB equations are usually difficult to solve. Here, because the particular forms of the couplings and the terminal conditions, we can solve the system by inspection, checking that a solution can be found in the form 
\begin{equation}
\label{fo:ansatzHJB}
V^i(t,x)=\frac{\eta_t}{2}(\o x -x^i)^2+\chi_t,
\end{equation}
for some deterministic scalar functions $t\mapsto \eta_t$ and $t\mapsto\chi_t$ satisfying $\eta_T=c$ and $\chi_T=0$ in order to match the terminal conditions for the $V^i$s. Indeed, the partial derivatives  $\partial_{x^j}V^i$ and $\partial_{x^jx^k}V^i$ read
\[
\partial_{x^j}V^i(t,x)=\eta_t \bigl(\frac{1}{N}-\delta_{i,j}\bigr)\left(\overline{x}-x^i\right),\quad \partial^2_{x^jx^k}V^i(t,x)=\eta_t\bigl( \frac{1}{N}-\delta_{i,j}\bigr)(\frac{1}{N}-\delta_{i,k}).
\]
and plugging these expressions into (\ref{fo:HJB}), and identifying term by term, we see that the system of HJB equations is solved if an only if
\begin{equation}
\label{fo:etat}
\begin{cases}
&\displaystyle \dot{\eta}_t=2(a+q)\eta_t+\bigl(1-\frac{1}{N^2}\bigr)\eta_t^2-(\epsilon-q^2),
\\
&\displaystyle \dot \chi_t= -\frac{1}{2}\sigma^2(1-\rho^2)\bigl(1-\frac{1}{N}\bigr)\eta_t,
\end{cases}
\end{equation}
with the terminal conditions $\eta_T=c$ and $\chi_T=0$.
As emphasized earlier, the Riccati equation is scalar and can be solved explicitly. One gets: 
\begin{equation}
\label{fo:eta-explicit}
\eta_{t}  = \frac{-(\epsilon-q^2)\bigl(e^{(\delta^+-\delta^-)(T-t)}-1\bigr)-c\bigl(\delta^+ e^{(\delta^+-\delta^-)(T-t)}-\delta^-\bigr)}
{\left(\delta^-e^{(\delta^+-\delta^-)(T-t)}-\delta^+\right)-c(1-1/N^2)\left(e^{(\delta^+-\delta^-)(T-t)}-1\right)},
\end{equation}
provided we set:
\begin{equation}
\label{fo:deltas}
\delta^{\pm}=-(a+q)\pm\sqrt{R}, \qquad \text{with}\qquad R=(a+q)^{2}+\left(1-\frac{1}{N^2}\right)(\epsilon-q^{2})>0.
\end{equation}
Observe that 
the denominator in \eqref{fo:eta-explicit} is always negative since
$\delta^+ > \delta^-$, so that $\eta_t$ is well defined for any $t\leq T$. The condition $q^2\leq \epsilon$ implies that $\eta_t$ is positive with $\eta_T=c$. Once $\eta_t$ is computed, 
one solves for $\chi_t$ (remember that $\chi_T=0$) and finds:
\begin{equation}
\label{fo:mu-explicit}
\chi_t=  \frac{1}{2}\sigma^2(1-\rho^2)\Bigl(1-\frac{1}{N}\Bigr)\int_t^T\eta_s\,ds.
\end{equation}
For the record, we note that the optimal strategies read 
\begin{equation}
\label{fo:opt_alpha}
\hat\alpha^i_{t}=q \bigl(\o X_t-X^i_t \bigr)-\partial_{x^i}V^i=\Bigl(q+(1-\frac{1}{N})\eta_t\Bigr)\bigl(\o X_t-X^i_t\bigr),
\end{equation}
and the optimally controlled dynamics:
\begin{equation}
\label{XicontrolledHJB}
dX^i_t=\Bigl(a+q+ (1-\frac{1}{N})\eta_t\Bigr)\bigl(\overline{X}_t-X^i_t\bigr)dt +\sigma \Bigl(\sqrt{1-\rho^2} dW^i_t+\rho dW^0_t\Bigr).
\end{equation}

\subsection{The Mean Field Limit}
We emphasize the dependence upon the number $N$ of players and we now write $\eta^N_t$ and $\chi^N_t$ for the solutions $\eta_t$ and $\chi_t$ of the system \eqref{fo:etat}, and $V^{i,N}(t,x)=(\eta^N/2) (\o x -x^i)^2+\chi_t^N$ for the value function of player $i$ in the $N$ player game. Clearly, 
$$
\lim_{N\to\infty}\eta^N_t=\eta^\infty_t,\qquad\text{ and }\qquad \lim_{N\to\infty}\chi^N_t=\chi^\infty_t,\
$$
where the functions $\eta^\infty_t$ and $\chi^\infty_t$ solve the system:
\begin{equation}
\label{fo:etatinfty}
\begin{cases}
&\displaystyle \dot{\eta}^\infty_t=2(a+q)\eta^\infty_t+(\eta^\infty_t)^2-(\epsilon-q^2),
\\
&\displaystyle \dot \chi^\infty_t= -\frac{1}{2}\sigma^2(1-\rho^2)\eta^\infty_t,
\end{cases}
\end{equation}
which is solved as in the case $N$ finite. We find
\begin{equation}
\label{fo:eta_infty}
\eta^\infty_{t}  = \frac{-(\epsilon-q^2)\bigl(e^{(\delta^+-\delta^-)(T-t)}-1\bigr)-c\bigl(\delta^+ e^{(\delta^+-\delta^-)(T-t)}-\delta^-\bigr)}
{\left(\delta^-e^{(\delta^+-\delta^-)(T-t)}-\delta^+\right)-c\left(e^{(\delta^+-\delta^-)(T-t)}-1\right)},
\end{equation}
and
\begin{equation}
\label{fo:mu_infty}
\chi^\infty_t=  \frac{1}{2}\sigma^2(1-\rho^2)\int_t^T\eta^\infty_s\,ds.
\end{equation}

\vskip 6pt
Next we consider the equilibrium behavior of the players' value functions $V^{i,N}$. 
For the purpose of the present discussion we notice that the value functions $V^{i,N}$ of all the players in the $N$ player game can be written as
$$
V^{i,N}\big(t,(x^1,\cdots,x^N)\big)=V^N\bigg(t,x^i,\frac1N\sum_{j=1}^N\delta_{x^j}\bigg)
$$
where the single function $V^N$ is defined as
$$
V^N(t,x,\mu)=\frac{\eta^N_t}{2}\bigg(x-\int_\RR xd\mu(x)\bigg)^2+\chi^N_t,\qquad (t,x,\mu)\in [0,T]\times\RR\times\cP_1(\RR),
$$
where ${\mathcal P}_{1}(\RR^d)$ denotes the space of integrable probability measures on $\RR^d$. 
Since the dependence upon the measure is only through the mean of the measure, we shall often use the function 
$$
v^N(t,x,m)=\frac{\eta^N_t}{2}(x-m)^2+\chi^N_t,\qquad (t,x,m)\in [0,T]\times\RR\times\RR,
$$
Notice that, at least for $(t,x,m)$ fixed, we have
$$
\lim_{N\to\infty}v^N(t,x,m)=v^\infty(t,x,m)
$$
where 
$$
v^\infty(t,x,m)=\frac{\eta^\infty_t}{2}(x-m)^2+\chi^\infty_t,\qquad (t,x,m)\in [0,T]\times\RR\times\RR.
$$
Similarly, all the optimal strategies in \eqref{fo:opt_alpha} may be expressed 
through a single feedback function $\hat{\alpha}^{N}(t,x,m) = [q+(1-1/N) \eta_{t}^N ](m-x)$
as $\hat{\alpha}_{t}^i = \hat{\alpha}^N(t,X_{t}^i,m_{t}^N)$. Clearly, 
\begin{equation*}
\lim_{N \rightarrow \infty } \hat{\alpha}^{N}(t,x,m) =  \hat{\alpha}^{\infty}(t,x,m),
\end{equation*} 
where
$\hat{\alpha}^{\infty}(t,x,m) = [q+\eta_{t}](m-x)$. 

Repeating the analysis in Subsection \ref{sub:spde},
we find that the limit of the empirical distributions satisfies the following version of \eqref{fo:francoismut}:
\begin{equation}
\label{ex:spde}
d\mu_t=-\partial_x\bigg([a(m_t-\,\cdot\,) - \alpha^{\infty}(t,\,\cdot\,)]\mu_t\bigg)dt +\frac{\sigma^2}{2} \partial^2_{x}\mu_t dt -\sigma\rho\partial_x\mu_t dW^0_t, \quad t \in [0,T],
\end{equation}
where $m_{t} = \int_{\RR^d} x d\mu_{t}(x)$, which is the Kolmogorov equation for the conditional marginal law 
given $W^0$ of the solution of the McKean-Vlasov equation:
\begin{equation}
\label{ex:MKV:SDE}
d \overline{X}_t=\left[a \bigl(m_t-\overline{X}_t)+\alpha^{\infty}(t,\overline{X}_t)\right]dt +\sigma \left(\rho dW^0_t+\sqrt{1-\rho^2}dW_t\right),
\end{equation}
subject to the condition $m_{t} = {\mathbb E}[\overline{X}_t \vert {\mathcal F}_{t}^0]$.  
Applying the Kolmogorov equation to the test function $\phi(x)=x$, we get
\begin{equation}
\label{fo:dmt}
dm_t= \bigg(\int \alpha^{\infty}(t,x) d\mu_t(x)\bigg) dt +\sigma\rho d W^0_t.
\end{equation}

We now write the stochastic HJB equation \eqref{fo:sHJB} in the present context. Remember that we assume that the stochastic 
flow $(\mu_t)_{0 \leq t \leq T}$ is given (as the solution of \eqref{ex:spde} with some prescribed initial condition $\mu_{0}=\mu$), 
and hence so is $(m_t)_{0 \leq t \leq T}$. Here
$$
L^*(t,x,y,z,z^0)=\inf_{\alpha\in A} \left[ [a(m_t-x)+\alpha]y+\frac{\sigma^2}{2}z+\sigma\rho z^0 +\frac{\alpha^2}2 - q\alpha (m_t-x)+\frac\epsilon 2 (m_t-x)^2\right].
$$
Since the quantity to minimize is quadratic in $\alpha$, we need to compute it for 
$
\bar\alpha=\bar{\alpha}(t,x,m_{t},y)$ with $\bar{\alpha}(t,x,m,y) = q(m -x) -y$.
 We get:
 $$
 L^*(t,x,y,z,z^0)=(a+q)(m_t-x)y - \frac12 y^2+\frac{\sigma^2}2 z +\sigma\rho z^0+\frac12 (\epsilon-q^2)(m_t-x)^2.
 $$
 Accordingly, the stochastic HJB equation takes the form
\begin{eqnarray}
\label{fo:risk_sHJB}
&&d_{t} V^{\mu}(t,x)=\bigg[ - (a+q)(m_t-x)\partial_xV^{\mu}(t,x) + \frac12 [\partial_xV^{\mu}(t,x)]^2 - \frac{\sigma^2}2 \partial^2_{x}V^{\mu}(t,x) \nonumber\\
&&\phantom{??????????????????} - \sigma\rho \partial_xZ^{\mu}(t,x) - \frac12 (\epsilon-q^2)(m_t-x)^2 \bigg]dt - Z^{\mu}(t,x) dW^0_t,
\end{eqnarray}
with the boundary condition $V^{\mu}_{t}(x)=(c/2)(m_{T}-x)^2$.

\subsection{Search for a Master Equation}
A natural candidate for solving
\eqref{fo:risk_sHJB} is the random field $(t,x) \mapsto v^{\infty}(t,x,m_{t})$, where as above 
$(m_{t})_{0 \leq t \leq T}$ denotes the means of the solution $(\mu_{t})_{0 \leq t \leq t}$
of the Kolmogorov SPDE \eqref{ex:spde}. This can be checked rigorously by using the expression of 
$v^{\infty}$ and by expanding $(v^{\infty}(t,x,m_{t}))_{0 \leq t \leq T}$ by It\^o's formula
(taking advantage of \eqref{fo:dmt}). As suggested at the end of the previous section, this shows that 
the stochastic HJB equation admits a solution $V^{\mu}(t,x)$ that can be expressed as a function of the current value 
$\mu_{t}$ of the solution of the Kolmogorov SPDE, namely
\begin{equation*}
V^{\mu}(t,x) = v^{\infty} \biggl(t,x,\int_{\RR^d} x' d\mu_{t}(x') \biggr). 
\end{equation*}
The same argument shows that $(\overline{X}_{t})_{0 \leq t \leq T}$ defined in \eqref{ex:MKV:SDE} as a solution of a McKean-Vlasov SDE is in fact
the optimal trajectory of the control problem considered in the item (ii) of the definition of a MFG, see 
\eqref{fo:mfgcontrolpb}, under the solution $(\mu_{t})_{0 \leq t \leq T}$ of the 
stochastic PDE \eqref{ex:spde}. Put it differently, $(\mu_{t})_{0 \leq t \leq T}$
is a solution of the MFG and the function $\alpha^{\infty}$ is the associated feedback control, as suggested by the asymptotic analysis performed in the previous paragraph. 

A natural question is to characterize the dynamics of the function $v^{\infty}$ 
in an intrinsic way. 
By definition of the value function (see \eqref{fo:stochastic_value:2}), we have
$$
V^{\mu}(t,\overline{X}_t)=\EE\bigg[\int_t^Tf\bigl(s,\overline{X}_s,\mu_s,\hat\alpha^{\infty}(s,\overline{X}_{s}) \bigr)ds
+ g\bigl( \overline{ X}_{T},\mu_{T}\bigr) \big|\cF_t \bigg]$$
so that 
$$
dV^{\mu}(t,\overline{X}_t)=-f\bigl(t,\overline{X}_t,\mu_t,\hat{\alpha}^{\infty}(t,\overline{X}_{t})\bigr)dt+ dM_t, \quad 
t \in [0,T],
$$
for some $(\cF_t)_{0 \leq t \leq T}$-martingale $(M_{t})_{0 \leq t \leq T}$. 
Recalling that $\bar{\alpha}(t,x,m,y)= q(m -x) -y$, $\partial_{x} v^{\infty}(t,x,m) = \eta_{t}^{\infty}(x-m)$,
and $\hat{\alpha}^{\infty}(t,x,m) = [q+\eta_{t}](m-x)$, we deduce that
\begin{equation*} 
\hat{\alpha}^{\infty}(t,x,m) = \bar{\alpha}\bigl(t,x,m,\partial_{x} v^{\infty}(t,x,m)\bigr),
\end{equation*}
which is the standard relationship in stochastic optimal control for expressing the optimal feedback in terms of the minimizer 
$\bar{\alpha}$ of the underlying extended Hamiltonian and of the gradient of the value function $v^{\infty}$. 
We deduce that 
\begin{equation*}
\begin{split}
f\bigl(t,\overline{X}_t,\mu_t,\hat{\alpha}^{\infty}(t,\overline{X}_{t})\bigr)
&= - \frac12 \bigl( q(m_{t}- \overline{X}_{t}) - \partial_{x} v^{\infty}(t,\overline{X}_{t},m_{t}) \bigr)
\bigl( q(m_{t}- \overline{X}_{t}) + \partial_{x} v^{\infty}(t,\overline{X}_{t},m_{t}) \bigr)
\\
&\hspace{15pt}+ \frac{\epsilon}2 \bigl( m_{t} - \overline{X}_{t} \bigr)^2,
\end{split}
\end{equation*}
so that
\begin{equation}
\label{fo:dVmu}
dV^{\mu}(t,\overline{X}_t)= \Big(- \frac12 (\epsilon-q^2)(m_t-\overline{X}_t)^2 - \frac12 \bigl[\partial_x v^{\infty}(t,\overline{X}_{t},m_{t}) \bigr]^2\Big)dt +dM_t.
\end{equation}
We are to compare this It\^o expansion with the It\^o expansion of $(v^{\infty}(t,\overline{X}_{t},m_{t}))_{0 \leq t \leq T}$.
Using the short-hand notation 
$v_t^{\infty}$ for $v^{\infty}(t, \overline{X}_t,m_t)$ and standard It\^o's formula, we get:
\begin{equation}
\label{fo:risk_sHJB:2}
\begin{split}
d v^{\infty}_{t}
&=\partial_tv_t^{\infty} dt+\partial_x v_t^{\infty} d\overline{X}_t + \partial_mv_t^{\infty} dm_t 
+\frac{\sigma^2}2 \partial^2_{xx}v^{\infty}_t +\frac{\sigma^2}2 \rho^2  \partial^2_{mm}v^{\infty}_t 
+\sigma^2\rho^2\partial^2_{xm}v_t^{\infty}
\\
&=\Big[\partial_t v_t^{\infty}+\partial_x v_t^{\infty} a(m_t-\overline{X}_t) +\partial_x v_t^{\infty} \hat{\alpha}^{\infty}
(t,\overline{X}_t)
+ \partial_m v_t^{\infty} \langle\mu_t,\alpha^{\infty}(t,\,\cdot\,)\rangle
\\
&\hspace{30pt}+\frac{\sigma^2}2 \partial^2_{x}v^{\infty}_t +\frac{\sigma^2}2 \rho^2  \partial^2_{m}v^{\infty}_t 
+\sigma^2\rho^2\partial^2_{xm}v^{\infty}_{t} \bigg] dt
\\
&\hspace{15pt} 
+ \sigma\rho[\partial_xv_t^{\infty} + \partial_m v_t^{\infty}] dW^0_t + \sigma\sqrt{1-\rho^2}\partial_xv_t^{\infty} dW_t.
\end{split}
\end{equation}
Identifying the bounded variation terms in \eqref{fo:dVmu} and \eqref{fo:risk_sHJB:2}, we get:
\begin{equation*}
\begin{split}
&\partial_t v_t^{\infty}+\partial_x v_t^{\infty} a(m_t-\overline{X}_t) +\partial_x v_t^{\infty} \hat{\alpha}^{\infty}
(t,\overline{X}_t)
+ \partial_m v_t^{\infty} \langle\mu_t,\alpha^{\infty}(t,\,\cdot\,)\rangle
\\
&\hspace{15pt} +\frac{\sigma^2}2 \partial^2_{x}v^{\infty}_t +\frac{\sigma^2}2 \rho^2  \partial^2_{m}v^{\infty}_t 
+\sigma^2\rho^2\partial^2_{xm}v^{\infty}_{t}
= - \frac12 (\epsilon-q^2)(m_t-\overline{X}_t)^2 - \frac12 \bigl[\partial_x v^{\infty}_{t}\bigr]^2,
\end{split}
\end{equation*}
where $\hat{\alpha}^{\infty}(t,x,m) = q(m-x) - \partial_{x} v^{\infty}(t,x,m)$. Therefore,
for a general smooth function $V:(t,x,m) \mapsto V(t,x,m)$, the above relationship with 
$v^{\infty}$ replaced by $V$ holds if
\begin{equation}
\label{fo:ME}
\begin{split}
&\partial_tV(t,x,m) +(a+q)(m-x)\partial_x V(t,x,m)  + \frac12 (\epsilon-q^2)(m-x)^2
- \frac12 [\partial_xV(t,x,m)]^2
\\
&\hspace{15pt} + \frac{\sigma^2}{2}\partial^2_{x}V(t,x,m) +\frac{\sigma^2}2 \rho^2  \partial^2_{m}V(t,x,m) 
+\sigma^2\rho^2\partial^2_{xm}V(t,x,m) = 0,
\end{split}
\end{equation}
for all $(t,x,m) \in [0,T] \times \RR^d \times \RR^d$
provided we have
\begin{equation}
\label{fo:restriction}
\int \partial_xV(t,x,m)d\mu(x)=0,\qquad\qquad 0\le t\le T,
\end{equation}
\eqref{fo:restriction} being used to get rid of the interaction between $\mu_{t}$ and $\alpha^{\infty}$. 
Obviously, $v^{\infty}$ satisfies \eqref{fo:restriction}. 
(Notice that this implies that the stochastic Kolmogorov equation becomes:
$dm_t=\rho\sigma dW^0_t$.)

Equation \eqref{fo:ME} reads as the dynamics for the decoupling field permitting to express 
the value function $V^{\mu}$ as a function of the current statistical state $\mu_{t}$ of the population. 
We call it the master equation of the problem.

\section{The Master Equation}
\label{se:ME}
While we only discussed mean field games so far, it turns out that the concept of master equation applies as well to the control of 
dynamics of McKean-Vlasov type whose solution also provides approximate equilibriums for large populations of individuals interacting through mean field
terms. See \cite{CarmonaDelarue_ap} for a detailed analysis. We first outline a procedure common to the two problems. Next we specialize this procedure to the
two cases of interest, deriving a master equation in each case. Finally, we highlight the differences to better understand what differentiates these
two related problems.
 
\subsection{General Set-Up}
Stated in loose terms, the problem is to minimize the quantity
\begin{equation}
\label{fo:cost}
\EE\bigg[\int_0^T f(s,X_s^{\alpha},\mu_s,\alpha_s)ds+g(X_T^{\alpha},\mu_T)\bigg]
\end{equation}
over the space of square integrable $\FF$-adapted controls $(\alpha_s)_{0\le s\le T}$ under the constraint that 
\begin{equation}
dX_s^{\alpha}=b\bigl(s,X_s^{\alpha},\mu_{s},\alpha_s\bigr)ds + 
\sigma(s,X_{s}^{\alpha},\mu_{s},\alpha_{s})dW_{s} + 
\int_{\Xi}\sigma^0(s,X_{s}^{\alpha},\mu_{s},\alpha_s,\xi)W^0(d\xi,ds).
\label{fo:MKV:SDE}
\end{equation}
Yet the notion of what we call a minimizer must be specified. Obvious candidates for a precise definition of the minimization problem lead to 
 different solutions. We consider two specifications: on the one hand, \textit{mean field games} and 
\textit{control of McKean-Vlasov dynamics} on the other. 
\vspace{2pt}

1. When handling mean-field games, minimization 
is performed along a frozen flow of measures $(\mu_{s}=\hat{\mu}_{s})_{0 \leq s \leq T}$ describing 
a statistical equilibrium of the population, and the stochastic process  
$(\hat{X}_{s})_{0\leq s \leq T}$ formed by the optimal paths of the optimal control problem \eqref{fo:cost}
is required to satisfy the matching constraints $\hat{\mu}_s= \cL(\hat{X}_s|\cF^0_s)$ for $0\le s\le T$. This is exactly the procedure
described in Subsection \ref{se:strategy}. 

2. Alternatively, minimization can be performed over the set of all the solutions 
of \eqref{fo:MKV:SDE} subject to the McKean-Vlasov constraint 
$(\mu_s=\mu_{s}^{\alpha})_{0\leq s \leq T}$, with
$\mu_{s}^{\alpha} = \cL(X_s^{\alpha}|\cF^0_s)$ for $0\le s\le T$, in which case the problem consists 
in minimizing the cost functional \eqref{fo:cost} over McKean-Vlasov diffusion processes. 
\vspace{2pt}

As  discussed painstakingly in \cite{CarmonaDelarueLaChapelle}, the two problems have different solutions 
since, in mean field games, the minimization is performed first and the fitting of the distribution of the optimal paths 
is performed next, whereas in the control of McKean-Vlasov dynamics, the McKean-Vlasov constraint is imposed first and
the minimization is handled next. Still, we show here that both problems can be reformulated in terms of master equations,
and we highlight the differences between the two in these reformulations. 

\vskip 4pt
The reason for handling both problems within a single approach is that in both cases, we rely on manipulations of a
\textit{value function} defined over the \textit{enlarged} state space
$\RR^d\times {\mathcal P}_{2}(\RR^d)$. For technical reasons, we restrict ourselves to measures in ${\mathcal P}_{2}(\RR^d)$ which denotes the space of 
square integrable probability measures (i.e. probability measures with a finite second moment). Indeed, for each $(t,x,\mu)\in [0,T]\times \RR^d\times {\mathcal P}_{2}(\RR^d)$,
we would like to define $V(t,x,\mu)$ as the expected future costs:
\begin{equation}
\label{fo:value function}
V(t,x,\mu) = \EE\bigg[\int_t^T f(s,X_s^{\hat{\alpha}},\hat{\mu}_s,\hat{\alpha}_s)ds+g(X_T^{\hat{\alpha}},\hat{\mu}_T)\big| X_{t}^{\hat{\alpha}} = x\bigg],
\end{equation}
where $\hat{\alpha}$ minimizes the quantity 
\eqref{fo:cost} when we add the constraint $\mu_t=\mu$ and compute the time integral between $t$ and $T$. In other words:
\begin{equation}
\label{fo:minimizer:value function}
(\hat{\alpha}_{s})_{t \leq s \leq T} = \textrm{argmin}_{\alpha} \EE\bigg[\int_t^T f(s,X_s^{\alpha},\mu_s,\alpha_s)ds+g(X_T^{\alpha},\mu_T) \bigg],
\end{equation}
the rule for computing  the infimum being as explained above, 
either from the mean field game procedure as in 1, 
or from the optimization over McKean-Vlasov dynamics as explained in 2. In both cases,
the flow $(\hat{\mu}_{s})_{t \leq s \leq T}$ appearing in \eqref{fo:value function} satisfies the fixed point condition 
$(\hat{\mu}_{s}={\mathcal L}(X_{s}^{\hat{\alpha}} \vert {\mathcal F}^{0,t}_{s}))_{t \leq s \leq T}$, 
which is true in both cases as $(X_{s}^{\hat{\alpha}})_{t \leq s \leq T}$ is an optimal path.
Here and in the following  $(\cF^{0,t}_{s})_{t\le s\le T}$ is the filtration generated by the future increments of the common noise $W^0$, in the sense that 
$\cF^{0,t}_{s} = \sigma \{W^0_r-W^0_t:\, \,t\le r\le s\}$. Recall that we use the notation $W^0_r$ for $\{W^0(\Lambda,[0,r)\}_\Lambda$ when $\Lambda$ varies through the Borel subsets of $\Xi$. Below, the symbol `hat' always refers to optimal quantities, and $(X_{s}^{\hat{\alpha}})_{t \leq s \leq T}$ is sometimes denoted by $(\hat{X}_{s})_{t \leq s \leq T}$. 

\vskip 6pt
Generally speaking, the definition of the (deterministic)  function $V(t,x,\mu)$ makes sense whenever the minimizer $(\hat{\alpha}_{s})_{t \leq s \leq T}$ exists 
and is unique. 
When handling mean-field games, some additional precaution is needed to guarantee the consistency of 
the definition. Basically, we also need that, given the initial distribution $\mu$ at time $t$, there exists a unique equilibrium flow of conditional probability measures 
$(\hat{\mu}_{s})_{t \leq s \leq T}$ satisfying $\hat{\mu}_{t}= \mu$ and $\hat{\mu}_{s} = {\mathcal L}(\hat{X}_{s} \vert {\mathcal F}^{0,t}_{s})$
for all $s \in [t,T]$,
where $(\hat{X}_{s})_{t \leq s \leq T}$ is the optimal path of the underlying minimization problem (performed under the fixed flow of measures $(\hat{\mu}_{s})_{t \leq s \leq T}$).
In that case, the minimizer $(\hat{\alpha}_{s})_{t \leq s \leq T}$
reads as the optimal control of $(\hat{X}_{s})_{t \leq s \leq T}$.
In the case of the optimal control of McKean-Vlasov 
stochastic dynamics, minimization is performed over the set of \emph{conditional McKean-Vlasov diffusion processes} with the prescribed initial distribution $\mu$ at time $t$, in other words, satisfying 
\eqref{fo:MKV:SDE} with ${\mathcal L}(X_{t}) = \mu$ and 
$\mu_{s} = \mu_{s}^{\alpha} = {\mathcal L}(X_{s}^{\alpha} \vert {\mathcal F}^{0,t}_{s})$ for all $s \in [t,T]$. 
In that case, the mapping $(t,\mu) \mapsto \int_{\RR^d} V(t,x,\mu) d\mu(x)$ appears as the value function of the optimal control problem:
\begin{equation}
\label{eq:MKV:opt:mean}
{\mathbb E} \bigl[ V(t,\chi,\mu) \bigr] = \inf_\alpha \EE\bigg[\int_t^{T} f\bigl(s,X_s^{\alpha},
{\mathcal L}(X_{s}^{\alpha} \vert {\mathcal F}_{s}^{0,t}),\alpha_{s}\bigr) ds
+g\bigl(X_{T}^{\alpha},{\mathcal L}(X_{T}^{\alpha} \vert {\mathcal F}_{T}^{0,t}) \bigr) 
\bigg],
\end{equation}
subject to $X_{t}^{\alpha}=\chi $ where $\chi$ is a random variable with distribution $\mu$, i.e. $\chi\sim \mu$.  

\vskip 6pt
Our goal is to characterize the function $V$ as the solution of a partial differential equation (PDE) on the space $[0,T] \times \RR^d\times\cP_2(\RR^d)$.
In the framework of mean-field games,  such an equation was touted in several presentations, and called the \emph{master equation}. See for example \cite{Lions}, \cite{Cardaliaguet} or \cite{GomesSaude}. We discuss the derivation of this equation below in Subsection \ref{sub:mfg}. 
Using a similar strategy, we also derive a master equation in the case of the optimal control of 
McKean-Vlasov 
stochastic dynamics in Subsection \ref{sub:mkv} below.

\subsection{Dynamic Programming Principle}
\label{subse:DPP}
In order to understand better the definition \eqref{fo:value function}, we consider the case in which 
the minimizer $(\hat{\alpha}_{s})_{t \leq s \leq T}$ has a feedback form, namely 
$\hat{\alpha}_{s}$ reads as $\hat{\alpha}(s,X_{s}^{\hat{\alpha}},\hat{\mu}_{s})$
for some function $\hat{\alpha} : [0,T] \times \RR^d \times {\mathcal P}_{2}(\RR^d) \rightarrow \RR$. In this case, 
\eqref{fo:value function} becomes
\begin{equation}
\label{fo:value function:2}
V(t,x,\mu) = \EE\bigg[\int_t^T f\bigl(s,X_s^{\hat{\alpha}},\hat{\mu}_s,\hat{\alpha}(s,X_{s}^{\hat{\alpha}},
\hat{\mu}_{s}) \bigr)ds+
g({X}_T^{\hat{\alpha}},\hat{\mu}_T) \big\vert X_{t}^{\alpha}= x \bigg],
\end{equation}
where $(X_{s}^{\hat{\alpha}})_{t \leq s \leq T}$ is the solution (if well-defined) of 
\eqref{fo:MKV:SDE} with $\alpha_{s}$ replaced by 
$\hat{\alpha}(s,X_{s}^{\hat{\alpha}},\hat{\mu}_{s})$. It is worth recalling that, in that writing, 
$\hat{\mu}_{s}$ matches the conditional law ${\mathcal L}(X_{s}^{\hat{\alpha}}\vert
{\mathcal F}_{s}^{0,t})$ and is forced to start from $\hat{\mu}_{t} =\mu$ at time $t$. 
\vskip 2pt

Following the approach used in finite dimension, a natural strategy is then to use \eqref{fo:value function:2} as a basis for 
deriving a dynamic programming principle for $V$. Quite obviously, a very convenient way to do so consists in requiring the
optimal pair  
$(\hat{X}_{s}=X_{s}^{\hat{\alpha}},\hat{\mu}_{s})_{t \leq s \leq T}$
to be Markov in $\RR^d \times {\mathcal P}_{2}(\RR^d)$, in which case
we get
\begin{equation*}
\begin{split}
& V(t+h,X_{t+h}^{\hat{\alpha}},\hat{\mu}_{t+h}) 
\\
&\hspace{15pt}= 
 \EE\bigg[\int_{t+h}^T f(s,X_s^{\hat{\alpha}},\hat{\mu}_s,\hat{\alpha}_s)ds
 +g(X_T^{\hat{\alpha}},\hat{\mu}_T) \bigl\vert 
 {\mathcal F}_{t+h}^{0,t} \vee \sigma\bigl\{X_t^{\hat{\alpha}},(W_{s}-W_{t})_{s \in [t,t+h]} \bigr\}
 \bigg].
\end{split}
\end{equation*} 
Here, the $\sigma$-field ${\mathcal F}_{t+h}^{0,t} \vee \sigma\{X_t^{\hat{\alpha}},(W_{s}-W_{t})_{s \in [t,t+h]}\}$
comprises all the events observed up until time $t+h$.

The rigorous proof of the 
Markov property for the path 
$(\hat{X}_{s}=X_{s}^{\hat{\alpha}},\hat{\mu}_{s})_{t \leq s \leq T}$
is left open. Intuitively, it sounds reasonable to expect that the Markov property holds if, for any initial 
distribution $\mu$, there exists a unique equilibrium $(\hat{\mu}_{s})_{t \leq s \leq T}$
starting from $\hat{\mu}_{t}=\mu$ at time $t \in [0,T]$. The reason is that, when uniqueness holds, there is no need to investigate the past of the optimal path in order to decide of the future of the dynamics. Such an argument is somehow quite generic in probability theory. In particular, the claim is expected to be true in both cases, whatever the meaning of what an equilibrium is. Of course,
this suggests that the following dynamic version of \eqref{fo:value function}
\begin{equation}
\label{fo:DPP:MKV}
V(t,x,\mu) = \EE\bigg[\int_t^{t+h} f(s,X_s^{\hat{\alpha}},\hat{\mu}_s,\hat{\alpha}_s)ds+
V\bigl(t+h,X_{t+h}^{\hat{\alpha}},\hat{\mu}_{t+h}\bigr) \bigl\vert X_{t}^{\hat{\alpha}} = x\bigg]
\end{equation}
must be valid. The fact that \eqref{fo:DPP:MKV} should be true in both cases
is the starting point for our common analysis of the master equation. For instance, as a by-product of \eqref{fo:DPP:MKV}, 
we can derive a variational form of the dynamic programming principle:
\begin{equation}
\label{fo:DPP:weak}
{\mathbb E} \bigl[ V(t,\chi,\mu) \bigr] = \inf \EE\bigg[\int_t^{t+h} f(s,X_s^{\alpha},\mu_s,\alpha_s)ds+V(t+h,X_{t+h}^{\alpha},\mu_{t+h}) \bigg],
\end{equation}
which must be true in both cases as well, provided the random variable $\chi$ has distribution $\mu$, i.e. $\chi \sim \mu$, the minimization being defined as above according to the situation.
  
The proof of \eqref{fo:DPP:weak} is as follows. First, we observe from \eqref{fo:value function:2}
that \eqref{fo:DPP:weak} must be valid when $t+h=T$. Then, \eqref{fo:DPP:MKV}
implies that the left-hand side is greater than the ride-hand side by choosing 
$(\hat{\alpha}_{s})_{t \leq s \leq T}$ as a control. To prove the converse inequality, we choose an 
arbitrary control $(\alpha_{s})_{t \leq s \leq t+h}$ between times $t$ and $t+h$. In the control of McKean-Vlasov 
dynamics, this means that the random measures $(\mu_{s})_{t \leq s \leq t+h}$ are chosen accordingly, 
as they depend on $(\alpha_{s})_{t \leq s \leq t+h}$, so that $\mu_{t+h}$ is equal to 
the conditional law of $X_{t+h}^{\alpha}$ at time $t+h$. At time $t+h$, this permits to 
switch to the optimal strategy
starting from $(X_{t+h}^{\alpha},\mu_{t+h})$. The resulting strategy
is of a greater cost than the optimal one. 
By \eqref{fo:value function:2}, this cost is exactly given by the right-hand side in \eqref{fo:DPP:weak}. 

\vskip 2pt
In the framework of mean field games, the argument for proving that the left-hand side is less than the right-hand side 
in \eqref{fo:DPP:weak}
is a bit different. The point is that
the flow $(\mu_{s})_{t \leq s \leq T}$ is fixed and matches $(\hat{\mu}_{s})_{t \leq s \leq T}$, so that 
$\hat{\alpha}(s,X_{s}^{\hat{\alpha}},\hat{\mu}_{s})$ reads as an optimal control for optimizing
\eqref{fo:cost}
in the \textit{environment} $(\mu_{s}=\hat{\mu}_{s})_{t \leq s \leq T}$.
So in that case, 
$V(t,x,\mu)$ is expected to match the optimal conditional cost
\begin{equation}
\label{fo:value function:MFG}
V(t,x,\mu) = \inf \EE\bigg[\int_t^T f(s,X_s^{\alpha},\hat{\mu}_s,\alpha_s)ds+g(X_T^{\alpha},\hat{\mu}_T) \bigl\vert X_{t}^{\alpha} = x\bigg],
\end{equation}
where $(X_{s}^{\alpha})_{t \leq s \leq T}$ solves the SDE 
\eqref{fo:MKV:SDE} with $(\mu_{s}=\hat{\mu}_{s})_{t \leq s \leq T}$ therein. 
Going back to \eqref{fo:DPP:weak}, the choice of an 
arbitrary control $(\alpha_{s})_{t \leq s \leq t+h}$ between times $t$ and $t+h$
doesn't affect the value of 
$(\mu_{s})_{t \leq s \leq t+h}$, which remains equal to 
$(\hat{\mu}_{s})_{t \leq s \leq t+h}$. 
At time $t+h$, this permits to 
switch to the optimal strategy
starting from $X_{t+h}^{\alpha}$ in the environment $(\hat{\mu}_{s})_{t \leq s \leq T}$. Again, the resulting strategy
is of a greater cost than the optimal one and, by \eqref{fo:value function:2}, this cost is exactly given by the right-hand side in \eqref{fo:DPP:weak}. 
\vskip 2pt

We emphasize that, when controlling McKean-Vlasov dynamics, 
\eqref{fo:value function:MFG} fails as in that case, the flow of measures is not frozen during
the minimization procedure. 
%As a consequence, the optimal flow of measures under the initial condition 
%$X_{t}=x$, denoted by $(\hat{\mu}_{s}^x)_{t \leq s \leq T}$, cannot be the same as the optimal flow of measures $%(\hat{\mu}_{s}^{\mu})_{t \leq s \leq T}$ under the initial condition 
%$X_{t} \sim \mu$.
In particular, the fact that 
\eqref{fo:value function:MFG} holds true in mean-field games only
suggests that 
$V$ satisfies a stronger dynamic programming principle in that case:
\begin{equation}
\label{fo:DPP:MFG}
V(t,x,\mu) = \inf \EE\bigg[\int_t^{t+h} f(s,X_s^{\alpha},\hat{\mu}_s,\alpha_s)ds+
V\bigl(t+h,X_{t+h}^{\alpha},\hat{\mu}_{t+h}\bigr) \bigl\vert X_{t}^{\alpha} = x\bigg].
\end{equation}
The reason is the same as above. On the one hand, 
\eqref{fo:DPP:MKV}
implies that the left-hand side is greater than the ride-hand side by choosing 
$(\hat{\alpha}_{s})_{t \leq s \leq T}$ as a control. On the other hand, 
choosing an arbitrary control 
$(\hat{\alpha}_{s})_{t \leq s \leq t+h}$ between $t$ and $t+h$
and switching to the optimal control 
starting from $X_{t+h}^{\alpha}$ in the environment $(\hat{\mu}_{s})_{t \leq s \leq T}$, the 
left-hand side must be less than the right-hand side. 
\subsection{Derivation of the Master Equation}
\label{subse:derivation}

As illustrated earlier (see also the discussion of the second example below), the derivation of the master equation can be based 
on a suitable chain rule for computing the dynamics of $V$ along paths of the form \eqref{fo:MKV:SDE}. 
This requires $V$ to be smooth enough in order to apply an \emph{It\^o like formula}. 

In the example tackled in the previous section, the dependence of $V$ upon 
the measure reduces to a dependence upon the mean of the measure, and 
a standard version of It\^o's formula could be used. In general, the measure argument lives 
in infinite dimension and different tools are needed. The approach advocated by P.L. Lions in his lectures at the 
\textit{Coll\`ege de France} suggests to \textit{lift-up}
the mapping $V$ into 
\begin{equation*}
\tilde{V} :  [0,T]\times\RR^d\times L^2(\tilde{\Omega},\tilde{\mathcal F},\tilde{\mathbb P};\RR^d)\ni (t,x,\tilde{\chi})
\mapsto \tilde{V}(t,x,\tilde{\chi}) = V(t,x,{\mathcal L}(\tilde{\chi})),
\end{equation*}
where $(\tilde{\Omega},\tilde{\mathcal F},\tilde{\mathbb P})$ can be viewed as a copy of 
the space $(\Omega,\mathcal F,\PP)$. The resulting $\tilde{V}$ is defined on the product of $[0,T]\times\RR^d$ and
a Hilbert space, for which the standard notion of Fr\'echet differentiability can be used. Demanding $V$ to be smooth in the measure 
argument is then understood as demanding $\tilde{V}$ to be smooth in the Fr\'echet sense. 
In that perspective, expanding $(V(s,X_{s}^{\alpha},\mu_{s}))_{t \leq s \leq T}$ is then 
the same as expanding $(\tilde{V}(s,X_{s}^{\alpha},\tilde{\chi}_{s}))_{t \leq s \leq T}$, where the process 
$(\tilde{\chi}_{s})_{t \leq s \leq T}$ is an It\^o process with $(\mu_{s})_{t \leq s \leq T}$ as flow
of marginal conditional distributions (conditional on ${\mathcal F}^{0,t}$). 

The fact that we require $(\tilde{\chi}_{s})_{t \leq s \leq T}$ to have $(\mu_{s})_{t \leq s \leq T}$
as flow of marginal conditional distributions calls for some precaution in the construction of the lifting. 
A way to do just this consists in writing  $(\Omega,{\mathcal F},\PP)$ in the form 
$(\Omega^{0} \times \Omega^1,{\mathcal F}^0 \otimes {\mathcal F}^1,\PP^0 \otimes \PP^1)$, 
$(\Omega^0,{\mathcal F}^0,\PP^0)$ supporting the common noise $W^0$,
and $(\Omega^1,{\mathcal F}^1,\PP^1)$ the idiosyncratic noise $W$. 
So an element $\omega \in \Omega$ can be written as $\omega=(\omega^0,\omega^1)
\in \Omega^0 \times \Omega^1$. Considering a copy $(\tilde{\Omega}^1,\tilde{\mathcal F}^1,
\tilde\PP^1)$ of the space $(\Omega^1,\mathcal F^1,
\PP^1)$, it then makes sense to consider the process $(\tilde{\chi}_{s})_{t \leq s \leq T}$
as the solution of an equation of the same form 
of \eqref{fo:MKV:SDE}, but on the space
$(\Omega^{0} \times \tilde{\Omega}^1,{\mathcal F}^0 \otimes \tilde{\mathcal F}^1,\PP^0 \otimes \tilde\PP^1)$,
$(\tilde{\Omega}^1,\tilde{\mathcal F}^1,\tilde\PP^1)$ being 
endowed with a copy $\tilde{W}$ of $W$.  The realization at some $\omega^0 \in \Omega^0$ of the 
conditional law of 
$\tilde{\chi}_{s}$ given ${\mathcal F}^0$ then reads as the law of the random variable 
$\tilde{\chi}_{s}(\omega^0,\cdot) \in L^2(\tilde{\Omega}^1,\tilde{\mathcal F}^1,\tilde{\PP}^1;\RR^d)$. 
Put in our framework,
this makes rigorous the identification of ${\mathcal L}(\tilde{\chi}_{s}(\omega^0,\cdot))$ with 
$\mu_{s}(\omega^0)$.

\vskip 6pt\noindent
Generally speaking, we expect that $(\tilde{V}(s,X_{s}^{\alpha},\tilde{\chi}_{s}) = 
\tilde{V}(s,X_{s}^{\alpha}(\omega^0,\omega^1),\tilde{\chi}_{s}(\omega^0,\cdot)))_{t \leq s \leq T}$
can be expanded as
\begin{equation}
\label{fo:Ito}
\begin{split}
d \tilde{V}\bigl(s,X_{s}^{\alpha},\tilde{\chi}_{s}\bigr)
&= \bigl[ \partial_{t} \tilde{V}(s,X_{s}^{\alpha},\tilde{\chi}_{s})
  + A_{x}^{\alpha} \tilde{V}(s,X_{s}^{\alpha},\tilde{\chi}_{s})
  + A_{\mu}^{\alpha} \tilde{V}(s,X_{s}^{\alpha},\tilde{\chi}_{s})
\\
&\hspace{15pt}  + A_{x \mu}^{\alpha} \tilde{V}(s,X_{s}^{\alpha},\tilde{\chi}_{s}) \bigr] ds
   + d M_s, \quad t \leq s \leq T,
\end{split}
  \end{equation}
  with $\t V(T,x,\t \chi) = g(x,{\mathcal L}(\t \chi))$ as terminal condition, 
where 
\begin{enumerate}
\item[$(i)$] $A_{x}^{\alpha}$ denotes the second-order differential operator associated to the process 
$(X_{s}^{\alpha})_{t \leq s \leq T}$.
It acts on functions of the state variable $x \in \RR^d$
and thus on the variable $x$ in $\tilde{V}(t,x,\tilde{\chi})$ in \eqref{fo:Ito}.
\item[$(ii)$] $A_{\mu}^{\alpha}$ denotes some second-order differential operator
associated to the process $(\tilde{\chi}_{s})_{t \leq s \leq T}$. It acts on functions from 
$L^2(\tilde{\Omega}^1,\tilde{\mathcal F}^1,\tilde{\PP}^1;\RR^d)$ into $\RR$ 
and thus on the variable $\tilde{\chi}$
in $\tilde{V}(t,x,\tilde{\chi})$.
\item[$(iii)$] $A_{x\mu}^{\alpha}$ denotes some second-order differential operator
associated to the cross effect of 
$(X^\alpha_{s})_{t \leq s \leq T}$
and $(\tilde{\chi}_{s})_{t \leq s \leq T}$, as both feel the same noise $W^0$. It 
acts on functions from 
$\RR^d \times L^2(\tilde{\Omega}^1,\tilde{\mathcal F}^1,\tilde{\PP}^1;\RR^d)$ and thus on the variables $(x,\tilde{\chi})$
in $\tilde{V}(t,x,\tilde{\chi})$.
\item[$(iv)$] $(M_{s})_{t \leq s \leq T}$ is a martingale. 
\end{enumerate}
A proof of \eqref{fo:Ito} is given in the appendix at the end of the paper. Observe that $A_{x\mu} \equiv 0$ if there is no common noise 
$W^0$.
Plugging \eqref{fo:Ito} into \eqref{fo:DPP:weak} and letting $h$ tend to $0$, we then expect:
\begin{equation}
\label{eq:full master PDE:0}
\partial_{t} {\mathbb E} \bigl[ 
\tilde{V}(t,\chi,\tilde{\chi}) \bigr] 
+ 
\inf_{\alpha} {\mathbb E}\bigl[ A_{x}^{\alpha} \tilde{V}(t,\chi,\tilde{\chi}) + A_{\mu}^{\alpha} \tilde{V}(t,\chi,\tilde{\chi})
+ A_{x\mu}^{\alpha} \tilde{V}(t,\chi,\tilde{\chi}) + f(t,\chi,\mu,\alpha) \bigr] = 0,
\end{equation}
where $\chi$ are $\tilde{\chi}$ random variables defined on 
$({\Omega}^1,{\mathcal F}^1,\PP^1)$ and
$(\tilde{\Omega}^1,\tilde{\mathcal F}^1,\t\PP^1)$ respectively,
both being 
distributed according to $\mu$. 
If the minimizer has a feedback form, namely if the optimization over 
random variables $\alpha$ reduces to optimization over
random variables of the form $\hat{\alpha}(t,\chi,\mu)$,  
$\hat{\alpha}$ being a function defined 
on $[0,T] \times \RR^d \times {\mathcal P}_{2}(\RR^d)$, then
the same strategy applied to \eqref{fo:DPP:MKV}, shows that $\tilde{V}$ satisfies the master equation
\begin{equation}
\label{eq:full master PDE}
\begin{split}
&\partial_{t}  
\tilde{V}(t,x,\tilde{\chi}) 
+ 
A_{x}^{\hat{\alpha}(t,x,\mu)} \tilde{V}(t,\chi,\tilde{\chi}) + A_{\mu}^{\hat{\alpha}(t,x,\t \mu)} \tilde{V}(t,\chi,\tilde{\chi})
+ A_{x\mu}^{\hat{\alpha}(t,x,\mu)} \tilde{V}(t,\chi,\tilde{\chi}) 
\\
&\hspace{75pt}+ f\bigl(t,\chi,\mu,\hat{\alpha}(t,x,\mu)\bigr) = 0.
\end{split}
\end{equation}

Of course, the rule for computing the infimum in \eqref{eq:full master PDE:0} depends on the framework. 
In the case of the optimal control of McKean-Vlasov diffusion processes, 
$(\tilde{\chi}_{s}(\omega^0,\tilde{\omega}^1))_{t \leq s \leq T}$
in \eqref{fo:Ito}
is chosen as a copy, denoted by $(\tilde{X}_{s}^{\alpha}(\omega^0,\tilde{\omega}^1))_{t \leq s \leq T}$, 
of $(X_{s}^{\alpha}(\omega^0,\omega^1))_{t \leq s \leq T}$ on the space $(\Omega^0 \times 
\tilde{\Omega}^1,{\mathcal F}^0 \otimes \tilde{\mathcal F}^1, \PP^0 \otimes \tilde{\PP}^1)$.
In that case, $A_{\mu}^{\alpha}$ depends on $\alpha$ explicitly.  
In the framework of mean field games, $(\tilde{\chi}_{s}(\omega^0,\tilde{\omega}^1))_{t \leq s \leq T}$ is chosen as a copy of the optimal path $(\hat{X}_{s})_{t \leq s \leq T}$
of the optimization problem \eqref{fo:value function} under the statistical equilibrium flow initialized at $\mu$
at time $t$. It does not depend on $\alpha$ so that $A_{\mu}^{\alpha}$ does not depend on $\alpha$. Therefore, $A_{\mu}=A_{\mu}^{\alpha}$ has no 
role in the computation of the infimum. 

\vskip 6pt
For the sake of illustration, we specialize the form of \eqref{eq:full master PDE} to a simpler case when 
\eqref{fo:MKV:SDE} reduces to 
$$
dX_s=b(s,X_s,\mu_s,\alpha_s)ds+\sigma(X_s)dW_s +\sigma^0(X_s)dW^0_s. 
$$ 
In that case, we know from the results presented in the appendix that 
\begin{equation}
\label{eq:operators}
\begin{split}
&A_{x}^{\alpha}  \t \varphi(t,x,\t \chi)
=  \langle b\bigl(t,x,{\mathcal L}(\t \chi),\alpha\bigr), \partial_{x} \t \varphi(t,x,\t \chi) \rangle 
\\
&\hspace{15pt}+
\frac12 {\rm Trace} \bigl[ \sigma(x) \bigl( \sigma(x) \bigr)^{\dagger} \partial_{x}^2 \t \varphi(t,x,\t \chi) \bigr]
+ \frac12 {\rm Trace} \bigl[  
\sigma^0 (x) \bigl( \sigma^{0}(x) \bigr)^{\dagger} \partial_{x}^2 \t \varphi(t,x,\t \chi) \bigr], 
\\
&A_{\mu}^{\alpha}  \t \varphi(t,x,\t \chi)
=   b\bigl(t,\t \chi,{\mathcal L}(\t \chi),\t \beta\bigr) \cdot D_{\mu} \t \varphi(t,x,\t \chi) 
\\
&\hspace{15pt}+
\frac12  
D^2_{\mu} \tilde{\varphi}\bigl(t,x,\t \chi\bigr) \bigl[ 
\sigma^0(\t \chi),
\sigma^0(\t \chi)
\bigr] + \frac12
D^2_{\mu} \tilde{\varphi}\bigl(t, x,\t \chi \bigr)  \bigl[ \sigma ( \t \chi) \t G,
\sigma ( \t \chi) \t G
\bigr],
\\
&A_{x\mu}^{\alpha}  \t \varphi(t,x,\t \chi)
= \langle \bigl\{ \partial_{x} D_{\mu}
\tilde{\varphi}\bigl(t, x,\t \chi\bigr) \cdot
\sigma^0(\t \chi) \bigr\},\sigma^0(x)
\bigr\rangle,
\end{split}
\end{equation}
where $\t G$ is an ${\mathcal N}(0,1)$ random variable on the space $(\t \Omega^1,\t {\mathcal F}^1,\t \PP^1)$, independent 
of $\t W$. The notations $D_{\mu}$ and $D^2_{\mu}$ refer to Fr\'echet derivatives of smooth functions on the space 
$L^2(\t \Omega^1,\t {\mathcal F}^1,\t \PP^1;\RR^d)$. 
For a random variable $\t \zeta \in L^2(\t \Omega^1,\t {\mathcal F}^1,\t \PP^1;\RR^d)$, 
the notation $D_{\mu} \t \varphi(t,x,\t \chi) \cdot \t \zeta$ denotes the action of the differential 
of $\t \varphi(t,x,\cdot)$ at point $\t \chi$ along the direction $\t \zeta$. Similarly, the notation 
$D_{\mu}^2 \t \varphi(t,x,\t \chi)  [\t \zeta,\t \zeta]$ denotes the action of the second-order 
differential 
of $\t \varphi(t,x,\cdot)$ at point $\t \chi$ along the directions $(\t \zeta,\t \zeta)$. 
We refer to the appendix for a more detailed account. 

Notice that $\t \chi$ in $A_{\mu}^{\alpha} \t \varphi(t,x,\t \chi)$ denotes the copy of $\chi$, 
$\chi$ standing for the value at time $t$ of the controlled
diffusion process $(\chi_{s})_{t \leq s \leq T}$. 
Specifying the value of $\chi$ according to the framework used 
for performing the optimization, we derive below the precise shape 
of the resulting master equation. Notice also that
 $A_{x\mu}^{\alpha} \t \varphi(t,x,\t \chi)$ does not depend upon $\alpha$ as the coefficients $\sigma^0$
 and $\sigma$ do not depend on it.

\subsection{The Case of Mean Field Games}
\label{sub:mfg}
In the framework of Mean-Field Games, $(\tilde{\chi}_{s})_{t \leq s \leq T}$ is chosen as a copy of the optimal path $(\hat{X}_{s})_{t \leq s \leq T}$. This says that, in \eqref{eq:operators}, 
$\t \chi$ stands for the value at time $t$ of the optimally controlled state from the optimization problem  \eqref{fo:value function} under the statistical equilibrium flow initialized at $\mu$ at time $t$. 
Therefore, the minimization in \eqref{eq:full master PDE:0}
reduces to 
\begin{equation}
\label{eq:minimization:MFG}
\begin{split}
&\inf_{\alpha} {\mathbb E} \bigl[ \langle 
b(t,\chi,\mu,\alpha),\partial_{x} \tilde{V}(t,\chi,\t \chi) \rangle + f(t,\chi,\mu,\alpha) \bigr]
\\
&=\inf_{\alpha} {\mathbb E} \bigl[ \langle 
b(t,\chi,\mu,\alpha),\partial_{x} V(t,\chi,\mu) \rangle + f(t,\chi,\mu,\alpha) \bigr],
\end{split}
\end{equation}
the equality following from the fact that $\partial_{x} \tilde{V}(t,x,\t \chi)$ is the same as 
$\partial_{x} V(t,x,\mu)$ (as the differentiation is performed in the component $x$).

Assume now that there exists a measurable mapping $\bar{\alpha} : [0,T] \times \RR^d 
\times {\mathcal P}_{2}(\RR^d) \times \RR^d \ni (t,x,\mu) \mapsto 
\bar{\alpha}(t,x,\mu,y)$, providing the argument of the minimization: 
\begin{equation}
\label{eq:optimizer}
\bar{\alpha}(t,x,\mu,y) = \textrm{arg}\;\inf_{\alpha \in \RR^d} H(t,x,\mu,y,\alpha) ,
\end{equation}
where the reduced Hamiltonian $H$ is defined as:
\begin{equation}
\label{fo:reduced}
H(t,x,\mu,y,\alpha) = 
 \langle b(t,x,\mu,\alpha), y\rangle
+ f(t,x,\mu,\alpha),
\end{equation}
Then, the minimizer in \eqref{eq:minimization:MFG}
must be $\alpha = \bar{\alpha}(t,\chi,\mu,\partial_{x} V(t,\chi,\mu))$, showing that 
$\hat{\alpha}(t,x,\mu)=
\bar{\alpha}(t,x,\mu,\partial_{x} V(t,x,\mu))$
is an optimal feedback. By
\eqref{eq:full master PDE}, the master equation reads
\begin{equation}
\label{eq:full master PDE:MFG:1}
\begin{split}
&\partial_{t} \tilde{V}(t,x,\tilde{\chi}) + 
\inf_{\alpha}  H\bigl(t,x,\mu,\partial_{x} \tilde{V}(t,x,\tilde{\chi}),\alpha\bigr)
  +  \bigl( A_{\mu} + A_{x \mu}
\bigr) \tilde{V}(t,x,\tilde{\chi})  
\\
 &\hspace{5pt}
+
\frac12 {\rm Trace} \bigl[ \sigma(x) \bigl( \sigma(x) \bigr)^{\dagger} \partial_{x}^2 \t V(t,x,\t \chi) \bigr]
+ \frac12 {\rm Trace} \bigl[  
\sigma^0 (x) \bigl( \sigma^{0}(x) \bigr)^{\dagger} \partial_{x}^2 \t V(t,x,\t \chi) \bigr]=0.
\end{split}
\end{equation}
By identification of the transport term, this says that 
the statistical equilibrium of the MFG with $\mu$ as initial distribution must be given by the 
solution of the conditional McKean-Vlasov equation:
\begin{equation}
\label{eq:MFG:opt}
d \hat{X}_{s} = b\bigl(s,\hat{X}_{s},\hat{\mu}_{s},\bar{\alpha}\bigl(s,\hat{X}_{s},\hat{\mu}_{s},
\partial_{x} V(s,\hat{X}_{s},\hat{\mu}_{s})
 \bigr) + \sigma\bigl( \hat{X}_{s} \bigr) dW_{s} + 
 \sigma^0 \bigl( \hat{X}_{s} \bigr) dW_{s}^0,
 \end{equation}
 subject to the constraint $\hat{\mu}_{s}
 ={\mathcal L}(\hat{X}_{s} \vert {\mathcal F}_{s}^0)$
 for $s \in [t,T]$,
 with $\hat{X}_{t} \sim \mu$.  
 We indeed claim

\begin{proposition}
\label{pr:mfg_decoupling}
On the top of the assumption and notation introduced right above, assume that, for all 
$t \in [0,T]$, $x \in \RR^d$ and 
$\mu \in {\mathcal P}_{2}(\RR^d)$
\begin{equation}
\label{eq:linear growth}
\vert \bar{\alpha}(t,x,\mu,y) \vert \leq C \biggl[ 1+ \vert x \vert + \vert y \vert + 
\biggl( \int_{\RR^d}
\vert x' \vert^2 d \mu(x') \biggr)^{1/2} \biggr],
\end{equation}
and that the growths of the coefficients $b$, $\sigma$ and $\sigma^0$ satisfy a similar bound. 
Assume also that 
$\t V$ is a (classical) 
solution of \eqref{eq:full master PDE:MFG:1} satisfying, for all 
$t \in [0,T]$, $x \in \RR^d$ and 
$\tilde{\chi}
\in L^2(\t \Omega^1,\t {\mathcal F}^1,\t \PP^1;\RR^d)$,
\begin{equation}
\label{eq:integrability}
\vert \partial_{x }\tilde{V}(t,x,\t \chi)
\vert + \Vert D_{\mu}\tilde{V}(t,x,\t \chi)
\Vert_{L^2(\t \Omega^1)}
\leq C \Bigl( 1 + \vert x \vert + \t \EE^1 \bigl[ \vert \t \chi \vert^2 \bigr]^{1/2}
\Bigr), 
\end{equation}
 and that, for any initial condition 
$(t,\mu) \in [0,T] \times {\mathcal P}_{2}(\RR^d)$, Equation \eqref{eq:MFG:opt} has a unique solution. 
Then, the flow $({\mathcal L}(\hat{X}_{s} \vert {\mathcal F}_{s}^0))_{t \leq s \leq T}$
solves the mean field game with $(t,\mu)$ as initial condition.  
\end{proposition}

\begin{proof}
The proof consists of a verification argument. 
First, we notice from \eqref{eq:linear growth} and \eqref{eq:integrability}
that the solution of \eqref{eq:MFG:opt} is square integrable in the sense that 
its supremum norm over $[0,T]$ is square integrable.
Similarly, for any square integrable control $\alpha$, the supremum of 
$X^{\alpha}$ (with $X_{t}^{\alpha} \sim \mu$) is square integrable.
The point is then to go back to \eqref{fo:value function:MFG} and to plug
$\hat{\mu}_{s} = {\mathcal L}(\hat{X}_{s} \vert {\mathcal F}_{s}^0)$ in the right-hand side. 
Replacing $g$ by $V(T,\cdot,\cdot)$ and applying It\^o's formula
in the appendix
(see Proposition \ref{prop:ito:joint}), using the growth and integrability assumptions to guarantee that
the expectation of the martingale part is zero, we deduce that the right-hand side 
is indeed greater than $V(t,x,\mu)$. Choosing 
$(\alpha_{s} = \bar{\alpha}(s,\hat{X}_{s},\hat{\mu}_{s},
\partial_{x} V(s,\hat{X}_{s},\hat{\mu}_{s}))_{t \leq s \leq T}$, equality must hold. This proves that 
$(\hat{X}_{s})_{t \leq s \leq T}$ is a minimization path 
of the optimization problem driven by its own flow of conditional distributions, 
which is precisely the definition of an MFG equilibrium.
\end{proof}

\begin{remark}
Proposition \ref{pr:mfg_decoupling} says that the solution of the master equation \eqref{eq:full master PDE:MFG:1} contains all the information needed to solve the mean field game problem.
In that framework, it is worth mentioning that the flow of conditional distributions 
$(\hat{\mu}_{s}={\mathcal L}(\hat{X}_{s} \vert {\mathcal F}_{s}^0))_{t \leq s \leq T}$
solves the SPDE \eqref{fo:spde}, with $\alpha(s,\cdot,\hat{\mu}_{s})
= \bar{\alpha}(s,x,\hat{\mu}_{s},
\partial_{x} V(s,x,\hat{\mu}_{s}))$. 
Notice finally that
$(Y_{s}=\partial_{x} V(s,\hat{X}_{s},\hat{\mu}_{s}))_{t \leq s \leq T}$
may be reinterpreted as the adjoint process in the stochastic Pontryagin principle
derived for mean field games in \cite{CarmonaDelarue_sicon} (at least when there is no common noise $W^0$).
In that framework, it is worth mentioning that the function 
$(t,x,\mu) \mapsto \partial_{x} V(t,x,\mu)$ reads as the decoupling field of the McKean-Vlasov FBSDE deriving from the stochastic Pontryagin principle. It plays the same role 
as the gradient  
of the value function in standard optimal control theory. See Subsection \ref{subse:viscosity}. 
\end{remark}

\subsection{The Case of the Control of McKean-Vlasov Dynamics}
\label{sub:mkv}
When handling the control of McKean-Vlasov dynamics, $(\tilde{\chi}_{s})_{t \leq s \leq T}$ is chosen as a copy of $(X_{s}^{\alpha})_{t \leq s \leq T}$. This says that, in \eqref{eq:operators}, 
$\t \alpha$ reads as a copy of $\alpha$ so that 
the minimization in \eqref{eq:full master PDE:0}
takes the form
\begin{equation*}
\begin{split}
&\inf_{\alpha} {\mathbb E} \bigl[ \langle 
b(t,\chi,\mu,\alpha),\partial_{x} \tilde{V}(t,\chi,\t \chi) \rangle + 
b(t,\t \chi,\mu,\t \alpha) \cdot D_{\mu} \tilde{V}(t,\chi,\t \chi)
+
f(t,\chi,\mu,\alpha) \bigr]
\\
&=\inf_{\alpha} {\mathbb E}^1 \Bigl[ \langle 
b(t,\chi,\mu,\alpha),\partial_{x} V(t,\chi,\mu) \rangle 
+ \t {\mathbb E}^1 \bigl[ 
\langle b(t,\t \chi,\mu,\t \alpha),\partial_{\mu} V(t,\chi,\mu)(\t \chi) \rangle
\bigr]
+ f(t,\chi,\mu,\alpha) \Bigr],
\end{split}
\end{equation*}
where the function $\partial_{\mu} V(t,x,\mu)(\cdot)$ \emph{represents} the Fr\'echet 
derivative $D_{\mu} \tilde{V}(t,x,\t \chi)$, that is 
$D_{\mu} \t V (t,x,\t \chi) = \partial_{\mu}
V(t,x,\mu)(\t \chi)$. See the appendix at the end of the paper for explanations. 
By Fubini's theorem, the minimization can be reformulated as
\begin{equation}
\label{eq:minimization:MKV}
\inf_{\alpha} {\mathbb E}^1 \Bigl[ \bigl\langle 
b(t,\chi,\mu,\alpha),\partial_{x} V(t,\chi,\mu) + \t \EE^1 \bigl[ 
\partial_{\mu} V(t,\t \chi,\mu)(\chi) 
\bigr]
\bigr\rangle 
+ f(t,\chi,\mu,\alpha) \Bigr].
\end{equation}
The strategy is then the same as in the previous paragraph. Assume indeed
that there exists a measurable mapping $\bar{\alpha} : [0,T] \times \RR^d 
\times {\mathcal P}_{2}(\RR^d) \times \RR^d \ni (t,x,\mu) \mapsto 
\bar{\alpha}(t,x,\mu,y)$ minimizing the reduced Hamiltonian 
as in \eqref{eq:optimizer},
then the minimizer in \eqref{eq:minimization:MKV}
must be 
\begin{equation*}
\begin{split}
\hat\alpha &= \bar{\alpha}\bigl(t,\chi,\mu,\partial_{x} V(t,\chi,\mu)
+ \t \EE^1[\partial_{\mu} V(t,\t \chi,\mu)(\chi)]\bigr)
\\
&= \bar{\alpha}\biggl(t,\chi,\mu,\partial_{x} V(t,\chi,\mu)
+ \int_{\RR^d} \partial_{\mu} V(t,x',\mu)(\chi) d\mu(x') \biggr),
\end{split}
\end{equation*} 
showing that
$\hat{\alpha}(t,x,\mu)=
\bar{\alpha}(t,x,\mu,\partial_{x} V(t,x,\mu)
+ \int_{\RR^d} \partial_{\mu} V(t,x',\mu)(x) d\mu(x'))$
is an optimal feedback. By \eqref{eq:full master PDE}, this permits to make explicit 
the form of the master equation. Notice that the term 
in $\hat{\alpha}$ in \eqref{eq:full master PDE}
does not read as an infimum, namely:
\begin{equation*}
\begin{split}
&\langle
b\bigl(t,x,\mu,\hat{\alpha}(t,x,\mu)\bigr),\partial_{x} V(t,x,\mu) \rangle
+ b\bigl(t,\t \chi,\mu,\hat{\alpha}(t,\t \chi,\mu)\bigr) \cdot D_{\mu} \t V
(t,x,\t \chi) + f\bigl(t,x,\mu,\hat{\alpha}(t,x,\mu)\bigr)
\\
&\not = \inf_{\alpha} \bigl[ b(t,x,\mu,\alpha) \cdot \partial_{x} \tilde{V}(t,x,\tilde{\chi})
+ b(t,x,\mu,\t \alpha\bigr) \cdot D_{\mu} \tilde{V}(t,x,\t \chi) 
+ f(t,x,\mu,\alpha) \bigr].
\end{split}
\end{equation*}
This says that the optimal path solving the optimal control of McKean-Vlasov dynamics must be given by:
\begin{equation}
\label{eq:MKV:opt}
\begin{split}
d \hat{X}_{s} &= b\biggl[s,\hat{X}_{s},\hat{\mu}_{s},\bar{\alpha}\biggl(s,\hat{X}_{s},\hat{\mu}_{s},
\partial_{x} V(s,\hat{X}_{s},\hat{\mu}_{s})
+ \int_{\RR^d}
\partial_{\mu} V(s,x',\hat{\mu}_{s})(\hat{X}_{s}) d \hat{\mu}_{s}(x')
 \biggr)\biggr] dt
 \\
 &\hspace{15pt}+ \sigma\bigl( \hat{X}_{s} \bigr) dW_{s} + 
 \sigma^0 \bigl( \hat{X}_{s} \bigr) dW_{s}^0,
 \end{split}
 \end{equation}
 subject to the constraint $\hat{\mu}_{s}
 ={\mathcal L}(\hat{X}_{s} \vert {\mathcal F}_{s}^0)$
 for $s \in [t,T]$,
 with $\hat{X}_{t} \sim \mu$.  
 We indeed claim

\begin{proposition}
\label{pr:mkv_decoupling}
On the top of the assumptions and notations introduced above, assume that 
$\bar{\alpha}$, $b$, $\sigma$ and $\sigma^0$ satisfy \eqref{eq:linear growth}. 
Assume also $\t V$ is a classical
solution of \eqref{eq:full master PDE} satisfying, for all 
$t \in [0,T]$, $x \in \RR^d$ and 
$\tilde{\chi}
\in L^2(\t \Omega^1,\t {\mathcal F}^1,\t \PP^1;\RR^d)$,
\begin{equation}
\label{eq:MKV:integrability}
\vert \partial_{x }\tilde{V}(t,x,\t \chi)
\vert + \Vert D_{\mu}\tilde{V}(t,x,\t \chi)
\Vert_{2,\t \Omega^1}
\leq C \Bigl( 1 + \vert x \vert + \t \EE^1 \bigl[ \vert \t \chi \vert^2 \bigr]^{1/2}
\Bigr), 
\end{equation}
 and that, for any initial condition 
$(t,\mu) \in [0,T] \times {\mathcal P}_{2}(\RR^d)$, Equation \eqref{eq:MKV:opt} has a unique solution.
Then, the flow $({\mathcal L}(\hat{X}_{s} \vert {\mathcal F}_{s}^0))_{t \leq s \leq T}$
solves the minimization problem \eqref{fo:cost}, set over controlled McKean-Vlasov dynamics.
\end{proposition}

\begin{proof}
The proof consists again of a verification argument. 
As for mean field games, 
we notice from \eqref{eq:linear growth} and \eqref{eq:MKV:integrability}
that the supremum (over $[0,T]$) of the solution of \eqref{eq:MKV:opt} is square integrable
and that, for any (square integrable) control $\alpha$, the supremum of 
$X^{\alpha}$ (with $X_{t}^{\alpha} \sim \mu$) is also square integrable.
The point is then to go back to \eqref{eq:MKV:opt:mean}.
Replacing $g$ by $V(T,\cdot,\cdot)$ and applying It\^o's formula
in the appendix
(see Proposition \ref{prop:ito:joint}) (taking benefit 
of the integrability condition \eqref{eq:integrability} for canceling 
the expectation of the martingale part) and using 
the same Fubini argument as in \eqref{eq:minimization:MKV}, 
we deduce that the right-hand side 
is indeed greater than $V(t,x,\mu)$.
Choosing
$\alpha_{s}=\hat{\alpha}(s,\hat{X}_{s},\hat{\mu}_{s})$, 
with $\hat{\alpha}(t,x,\mu)=
\bar{\alpha}(t,x,\mu,\partial_{x} V(t,x,\mu)
+ \int_{\RR^d} \partial_{\mu} V(t,x',\mu)(x) d\mu(x'))$, equality must hold.
\end{proof}

\begin{remark}
The flow of conditional distributions 
$(\hat{\mu}_{s}={\mathcal L}(\hat{X}_{s} \vert {\mathcal F}_{s}^0))_{t \leq s \leq T}$
solves an SPDE, on the same model as \eqref{fo:spde}. The formulation of that SPDE is left to the reader.  

Notice finally that
$(\partial_{x} V(s,\hat{X}_{s},\hat{\mu}_{s})
+ \int_{\RR^d}
\partial_{\mu} V(s,x,\hat{\mu}_{s})(\hat{X}_{s}) d\hat{\mu}_{s}(x)
)_{t \leq s \leq T}$
may be reinterpreted as the adjoint process in the stochastic Pontryagin principle
derived for the control of McKean-Vlasov dynamics in \cite{CarmonaDelarue_ap} (at least when there is no common noise $W^0$).
In particular, the function $(t,x,\mu) \mapsto 
\partial_{x} V(s,x,\mu)
+ \int_{\RR^d}
\partial_{\mu} V(s,x,\mu)(x) d\mu(x)$ reads as the decoupling field 
of the McKean-Vlasov FBSDE deriving from the stochastic Pontryagin principle
for the control of McKean-Vlasov dynamics. 
It is interesting to notice that the fact that the formula contains two different terms is a perfect reflection of the backward propagation of the terminal condition of the FBSDE. Indeed, as seen in \cite{CarmonaDelarue_ap}, this terminal condition has two terms corresponding to the partial derivatives of the terminal cost function $g$ with respect to the state variable $x$ and the distribution $\mu$. See Subsection \ref{subse:viscosity}.
\end{remark}
 
\subsection{Viscosity Solutions}
\label{subse:viscosity}
In the previous paragraph, we used the master equation within the context of a verification argument to identify optimal paths of the underlying optimal control problem, and we alluded to the connection with purely probabilistic methods derived from the stochastic Pontryagin principle. The stochastic Pontryagin principle works as follows: under suitable conditions, optimal paths are identified with the forward component of a McKean-Vlasov FBSDE. In that framework, our discussion permits to identify the gradient of the function $V$ with the decoupling field of the FBSDE. 
This FBSDE has the form:
\begin{equation}
\label{fo:fbsde}
\begin{cases}
& dX_s=b(s,X_s,\mu_s,\hat\alpha(s,X_s,\mu_s,Y_s))ds +\sigma^0(X_s)dW^0_s +\sigma(X_s)dW_s,
\\
& dY_s=- \Psi\bigl(s,X_s,\mu_s,\hat\alpha(s,X_s,\mu_s,Y_s)\bigr) ds+Z^0_s dW^0_s+Z_sdW_s, \qquad Y_T= \phi(X_T,\mu_T)
\end{cases}
\end{equation}
for some functions $(t,x,\nu,\alpha)\mapsto \Psi(t,x,\nu,\alpha)$
and $(x,\mu) \mapsto \phi(x,\mu)$, the McKean-Vlasov nature of the FBSDE being due to the constraints $\mu_s=\cL(X_s|\cF^0_s)$ and $\nu_s=\cL((X_s,Y_s)|\cF^0_s)$.

\vskip 2pt
In the mean field game case, the stochastic Pontryagin principle takes the form
\begin{equation}
\label{fo:MFGpsi}
\Psi(t,x,\nu,\alpha)=\partial_x H\bigl(t,x,\mu,y,\alpha\bigr),%_{|\alpha=\hat\alpha(t,x,\mu,y)},
\quad \phi(x,\mu) = \partial_{x} g(x,\mu), 
\end{equation}
where $\mu$ denotes the first marginal of $\nu$, and 
\begin{equation}
\label{fo:MKVpsi}
\begin{split}
&\Psi(t,x,\nu,\alpha)= \partial_xH\bigl(t,x,\mu,y,\alpha\bigr)%_{\vert \alpha=\hat\alpha(t,x,\mu,y)} + 
+ 
\int_{\RR^d\times\RR^d}\bigl( \partial_\mu H\bigl(t,x',\mu,y',\alpha'\bigr)(x) \bigr)_{\vert \alpha'=\hat\alpha(t,x',\mu,y')} \nu(dx',dy'),
%\t \EE^1\bigl[ \bigl( \partial_\mu H\bigl(t,\t X,\mu,\t Y,\alpha\bigr)(x) \bigr)_{\vert \alpha=\hat\alpha(t,\t X,\mu,\t Y)} \bigr],
\\
&\phi(x,\mu) = \partial_{x} g(x,\mu) + \int_{\RR^d}\partial_\mu g(x',\mu)(x)\mu(dx')
%\t\EE^1 \bigl[ \partial_{\mu} g(\t X_{T}^1,\mu)(x) \bigr].
\end{split}
\end{equation}
in the case of the control of McKean-Vlasov dynamics.

\vskip 6pt
One may wonder if a converse to the strategy discussed previously is possible: how could we reconstruct a solution of the master equation from a purely probabilistic approach? Put it differently, given the solution of the McKean-Vlasov FBSDE characterizing the optimal path \textit{via} the stochastic Pontryagin principle, is it possible to reconstruct $V$ and to prove that it satisfies a PDE or SPDE which we could identify to the \emph{master equation}? 

In the forthcoming paper \cite{ChassagneuxCrisanDelarue}, the authors investigate the differentiability of the flow of a McKean-Vlasov FBSDE and manage, in some cases, to reconstruct $V$ as a classical solution of the master equation. 

A more direct approach consists in constructing $V$ as a viscosity solution of the master equation. This direct approach was used in \cite{Cardaliaguet} for non-stochastic games. In all cases the fundamental argument relies on a suitable form of the dynamic programming principle. This was our motivation for the discussion in Subsection
\ref{subse:DPP}. Still we must remember that Subsection \ref{subse:DPP} remains mostly at the heuristic level, and that a complete proof of the dynamic programming principle in this context would require more work. This is where the stochastic Pontryagin principle may help. If uniqueness of the optimal paths and of the equilibrium are known (see for instance \cite{CarmonaDelarue_sicon} and \cite{CarmonaDelarue_ap}), then the definition of 
$V$ in \eqref{fo:value function} makes sense. In this case, not only do we have the explicit form the optimal paths, but the dynamic programming principle is expected to hold. 

We refrain from going into the gory details in this review paper. Instead, we take the dynamic programming principle for granted. The question is then to derive the master equation for $V$ in the viscosity sense, from the three possible versions 
\eqref{fo:DPP:MFG}, \eqref{fo:DPP:MKV} and \eqref{fo:DPP:weak}.
 In the present context, since differentiability with respect to one of the variables is done through a lifting of the functions, we will be using the following definition of viscosity solutions.
 
 \begin{definition}
 \label{def:viscosity}
 We say that $V$ is a super-solution (resp. sub-solution) in the sense of viscosity of the master equation if whenever $(t,x,\mu)\in [0,T]\times\RR^d\times\cP_2(\RR^d)$ and the function $[0,T]\times\RR^d\times\cP_2(\RR^d)\ni(s,y,\nu)\mapsto \varphi(s,y,\nu)$ is continuously differentiable, once in the time variable $s$, and twice in the variables $y$ and $\nu$, satisfies  $V(t,x,\mu)=\varphi(t,x,\mu)$ and $V(s,y,\nu)\ge\varphi(s,y,\nu)$ for all $(s,y,\nu)$ then we have \eqref{eq:full master PDE:0}
and/or  \eqref{eq:full master PDE}, with $\tilde{V}$ replaced by $\tilde{\varphi}$
and $=0$ replaced by $\leq 0$ (respectively by $\geq 0$). 
\end{definition}
The reason why we say \emph{and/or} might look rather strange. This will be explained below, the problem being actually more subtle than it seems at first. 

\vskip 4pt
Following the approach used in standard stochastic optimal control problems, 
the proof could consist in applying It\^o's formula to $\t \varphi(s,X^{\hat{\alpha}}_s,\hat{\mu}_s)_{t \leq s \leq t+h}$. In fact, there is no difficulty in proving 
the \emph{viscosity inequality} \eqref{eq:full master PDE} by means of \eqref{fo:DPP:MKV}. Still, this result is rather useless as the optimizer 
$\hat{\alpha}$ is expected to depend upon the gradient of $\t V$ and much more, as $\hat{\alpha}$
reads as $\bar{\alpha}$ applied to the gradient of $\t V$. 
The question is thus to decide whether it makes sense to replace the gradient of $\t V$
in $\bar{\alpha}$ by the gradient of $\t \varphi$. To answer the question, we must distinguish the two problems: 
\vskip3pt

1. In the framework of mean field games, the answer is yes. The reason is that, when 
$V$ is smooth, the inequality $V \geq \varphi$ in the neighborhood of $(t,x,\mu)$
implies $\partial_{x}V(t,x,\mu) = 
\partial_{x}\varphi(t,x,\mu)$. This says that we expect $\t \varphi$ to satisfy 
\eqref{eq:full master PDE:MFG:1} with $=0$ replaced by $\leq 0$. Actually, this can be checked rigorously by means of the stronger version \eqref{fo:DPP:MFG} of the dynamic programming principle, following the proof in \cite{FlemingSoner}. 
\vskip3pt

2. Unfortunately, this is false when handling the control of McKean-Vlasov dynamics. Indeed, the gradient of $V$ is then understood as
$\partial_{x} V(t,x,\mu) + \int_{\RR^d} \partial_{\mu} V(t,x',\mu)(x) d\mu(x')$, which is 
`non-local' in the sense that it involves values of $V(t,x',\mu)$ for $x'$ far away from $x$. 
In particular, there is no way one can replace 
$\partial_{x} V(t,x,\mu) + \int_{\RR^d} \partial_{\mu} V(t,x',\mu)(x) d\mu(x')$
by $\partial_{x} \varphi(t,x,\mu) + \int_{\RR^d} \partial_{\mu} \varphi(t,x',\mu)(x) d\mu(x')$
on the single basis of the comparison of $\varphi$ and $V$. 
This implies that, in the optimal control of McKean-Vlasov dynamics, viscosity solutions must be discussed in the framework of \eqref{eq:full master PDE:0}. Obviously, this requires adapting the notion of viscosity solution as only the function $(t,\mu) \mapsto 
\int_{\RR^d} V(t,x,\mu) d\mu(x)$ matters in the dynamic programming principle
\eqref{fo:DPP:weak}. Comparison is then done with test functions of the form
$(t,\mu) \mapsto \int_{\RR^d} \phi(t,x,\mu) d\mu(x)$ (or simply $\phi(t,\mu)$). 
The derivation of an inequality in \eqref{eq:full master PDE:0} is then achieved by a new application of 
It\^o's formula.

\subsection{Comparison of the Two Master Equations}

We repeatedly reminded the reader that the function $V$ obtained in the case of
mean field games (whether or not there is a common noise) \emph{is not a value function} in the usual sense of optimal control. Indeed, solving a mean field game problem is finding a fixed point problem more than solving an optimization problem. For this reason, the master equation should not read (and should not be interpreted) as a Hamilton-Jacobi-Bellman equation. Indeed, 
even though the first terms in Equation \eqref{eq:full master PDE:MFG:1} 
are of Hamiltonian type, the extra term $A_{\mu}$ (specifically the first order term in $A_{\mu}$) shows that this equation is not an HJB equation.
On the other hand, the previous subsection shows that the
 master equation for the control of McKean-Vlasov dynamics, which comes from an optimization problem, can be viewed as an HJB equation when put in the form 
\eqref{eq:full master PDE:0}. In that case, the solution reads as the value function $(t,\mu) \mapsto 
\int_{\RR^d} V(t,x,\mu) d\mu(x)$ of the corresponding optimization problem. 

\section{A Second Example: A Simple Growth Model}
\label{se:2ndexample}

The following growth model was introduced and studied in \cite{GueantLasryLions.pplnm}.
We review its main features by recasting it in the framework of the present discussion of the master equation of mean field games 
with common noise.
In fact the common noise $W^0$ is the only noise of the model since $\sigma\equiv 0$ and the idiosyncratic noises do not appear.
\subsection{Background}
As it is the case in many economic models, the problem in \cite{GueantLasryLions.pplnm} is set for an infinite time horizon ($T=\infty$) with a positive discount rate $r>0$.
As we just said, $\sigma\equiv 0$.  Moreover, the common noise is a one dimensional Wiener process $(W^0_t)_{ t\ge 0}$. As before, we denote by $\FF^0=(\cF^0_t)_{t\ge 0}$ its filtration. We also assume that its volatility is linear, that is $\sigma^0(x)=\sigma x$ for some positive constant $\sigma$, and that each player controls the drift of its state so that $b(t,x,\mu,\alpha)=\alpha$. In other words, the dynamics of the state of player $i$ read:
\begin{equation}
\label{SDE:pareto}
dX^i_t=\alpha^i_t dt+\sigma X^i_t dW_t^0.
\end{equation}
We shall restrict ourselves to Markovian controls of the form $\alpha^i_t=\alpha(t,X^i_t)$ for a deterministic function 
$(t,x)\mapsto \alpha(t,x)$, which will be assumed non-negative and Lipschitz in the variable $x$. Under these conditions, 
for any player, say player $1$, $X^1_{t} \geq 0$ at all times $t>0$ if $X^1_{0} \geq 0$ and
for any two players, say players $1$ and $2$, the homeomorphism property of Lipschitz SDEs implies that
$X^1_t\le X^2_t$ at all times $t>0$ if $X^1_0\le X^2_0$.

 Note that in the particular case
\begin{equation}
\label{fo:alpha_linear}
\alpha(t,x)=\gamma x
\end{equation}
for some $\gamma>0$, then
\begin{equation}
\label{fo:diff}
X^2_t=X^1_t+(X^2_0-X^1_0)e^{(\gamma - \sigma^2/2) t+\sigma W_t^0}.
\end{equation}
We assume that $k>0$ is a fixed parameter and we introduce a special notation for the family of scaled Pareto distributions with decay parameter $k$.
For any real number $q\ge 1$, we denote by $\mu^{(q)}$ the Pareto distribution:
\begin{equation}
\label{fo:Pareto}
\mu^{(q)}(dx)=k\frac{q^k}{x^{k+1}}\bone_{[q,\infty)}(x)dx.
\end{equation}
Notice that $X\sim \mu^{(1)}$ is equivalent to $qX\sim\mu^{(q)}$. We shall use the notation $\mu_t$ for the conditional distribution of the state $X_t$ of a generic player
at time $t\ge 0$ conditioned by the knowledge of the past up to time $t$ as given by $\cF^0_t$. 
Under the prescription \eqref{fo:alpha_linear}, 
we claim that, if $\mu_0=\mu^{(1)}$, then $\mu_t=\mu^{(q_t)}$ 
where $q_t=e^{(\gamma - \sigma^2/2) t+\sigma W_t^0}$. In other words, conditioned on the history of the common noise, the distribution of the states of the players remains Pareto with parameter $k$ if it started that way, and the left-hand point of the distribution $q_t$ can be understood as a sufficient statistic characterizing the distribution $\mu_t$. 
This remark is an immediate consequence of formula \eqref{fo:diff} applied to $X^1_t=q_t$, in which case $q_0=1$, and $X^2_t=X_t$, implying that $X_t=X_0 q_t$.
So if $X_0\sim\mu^{(1)}$, then $\mu_t\sim\mu^{(q_t)}$. In particular, we have an explicit solution of the conditional Kolmogorov equation in the case of the particular linear
feedback controls.

\subsection{Optimization Problem}
We now introduce the cost functions and define the optimization problem.
 We first assume that the problem is set for a finite horizon $T$. For the sake of convenience, we skip the stage of the $N$ player game for $N$ finite, and discuss directly the limiting MFG problem in order to avoid dealing with the fact that empirical measures do not have densities.
 The shape of the terminal cost $g$ will be specified later on. Using the same notation as in \cite{GueantLasryLions.pplnm}, we define the running cost function $f$ by
$$
f(x,\mu,\alpha)=  c \frac{x^a}{[(d\mu/dx)(x)]^b}-\frac{E}{p}\frac{\alpha^p}{[\mu([x,\infty))]^b},
$$
for some positive constants $a$, $b$, $c$, $E$ and $p >1$ whose economic meanings are discussed in \cite{GueantLasryLions.pplnm}. We use the convention that the density is the density of the absolutely continuous part of the Lebesgue's decomposition of the measure $\mu$, and that in the above sum, the first term is set to $0$ when this density is not defined or is itself $0$. The extended Hamiltonian 
of the system (see \eqref{eq:optimizer}) reads
$$
H(x,y,\mu,\alpha)=\alpha y + c \frac{x^a}{[(d\mu/dx)(x)]^b} - \frac{E}{p}\frac{\alpha^p}{[\mu([x,\infty))]^b}
$$
and the value $\bar\alpha$ of $\alpha$ minimizing $H$ is given by (for $y \geq 0$):
\begin{equation}
\label{fo:hat_alpha}
\bar\alpha=\bar\alpha(x,\mu,y)=\bigg(  \frac{y}E \bigl[ \mu([x,\infty)) \bigr]^b\bigg)^{1/(p-1)}
\end{equation}
so that:
\begin{equation*}
\begin{split}
H(x,y,\mu,\bar\alpha)&= \bigg( \frac{y}E \bigl[ \mu([x,\infty)) \bigr]^b\bigg)^{1/(p-1)}y + c \frac{x^a}{
[(d\mu/dx)(x)]^b}
\\
&\hspace{15pt}
-\frac{E}{p}\frac{\Big( (y/E) [\mu([x,\infty))]^b\Big)^{p/(p-1)}}{[\mu([x,\infty))]^b}
\\
&=\frac{p-1}{p} E^{-1/(p-1)}y^{p/(p-1)} \bigl[ \mu([x,\infty)) \bigr]^{b/(p-1)} + c 
\frac{x^a}{[(d\mu/dx)(x)]^b}.
\end{split}
\end{equation*}
In the particular case of linear controls \eqref{fo:alpha_linear}, using the explicit formula \eqref{fo:Pareto} for the density of $\mu^{(q)}$ and the fact that
$$
\mu^{(q)}([x,\infty))=1\wedge \frac{q^k}{x^k},
$$
we get
\begin{equation*}
\begin{split}
f\bigl(x,\mu^{(q)},\alpha\bigr)&= c \frac{x^a}{(k q^k / x^{k+1})^b} 
{\mathbf 1}_{\{x \geq q\}}
- \frac{E}{p}\frac{\alpha^p}{1\wedge (q^{kb}/x^{kb})}
\\
&=
\frac{c}{k^bq^{kb}}x^{a+b(k+1)}
{\mathbf 1}_{\{x \geq q\}}
-\frac{E}{pq^{kb}}\alpha^p \bigl( x^{kb} \vee q^{kb} \bigr),
\end{split}
\end{equation*}
and
\begin{equation}
\label{fo:hat_alpha'}
\bar\alpha(x,\mu,y)=\bigg[\frac{y}E \Bigl( \frac{q^{kb}}{x^{kb}} \wedge 1 \Bigr)\bigg]^{1/(p-1)},
\end{equation}
so that
$$
H(x,y,\mu^{(q)},\bar\alpha)=\frac{p-1}{p} E^{-1/(p-1)} y^{p/(p-1)} \
\Bigl( \frac{q^{kb/(p-1)}}{x^{kb/(p-1)}} \wedge 1 \Bigr) 
+ c \frac{x^{a+(k+1)b}}{k^bq^{kb}} {\mathbf 1}_{\{x \geq q\}}.
$$
\subsection{Search for a Pareto Equilibrium}
Assuming that the initial distribution of the values of the state is given by the Pareto
distribution $\mu^{(1)}$, we now restrict ourselves in searching for equilibriums with Pareto distributions, which means that 
the description of the equilibrium flow of measures $(\hat{\mu}_{t})_{0 \leq t \leq T}$ can be reduced to the 
description of the flow of corresponding Pareto parameters $(\hat{q}_{t})_{0 \leq t \leq T}$. 
Introducing the letter $V$ for denoting the solution of the master equation, we know from 
\eqref{eq:full master PDE:MFG:1} 
and Proposition \ref{pr:mfg_decoupling} that the optimal feedback control must read
\begin{equation*}
\hat{\alpha}(t,x) = \bar{\alpha} \bigl(x,\hat{\mu}_{t},\partial_{x} V(t,x,\hat{\mu}_{t})\bigr)
= \bigg[\frac{\partial_{x} V(t,x,\hat{\mu}_{t})}E 
\Bigl( \frac{\hat{q}_{t}^{kb}}{x^{kb}} \wedge 1 \Bigr)\bigg]^{1/(p-1)}.
\end{equation*}
In order to guarantee that the equilibrium flow of measures is of Pareto type, it must satisfy the condition:
\begin{equation}
\label{Cond:Pareto}
\gamma x = \bigg(\frac{\partial_{x} V(t,x,\hat{\mu}_{t})}E \frac{\hat{q}_{t}^{kb}}{x^{kb}}\bigg)^{1/(p-1)},
\quad x \geq \hat{q}_{t}. 
\end{equation}
for some $\gamma >0$. There is no need for checking the condition for $x < \hat{q}_{t}$ as the
path driven by the Pareto distribution is then always greater than or equal to $(\hat{q}_{t})_{t \geq 0}$. 

Since we focus on equilibriums of Pareto type, we compute the function $V$ at distributions of Pareto type only. It then 
makes sense to 
\textit{parameterize} the problem and to seek for $V$ in the factorized form:
$$
\cV(t,x,q)=V(t,x,\mu^{(q)}),
$$
for some function $\cV : (t,x,q)\in [0,T]\times \RR \times \RR \rightarrow \RR$. 
Then, the relationship \eqref{Cond:Pareto} takes the form:
\begin{equation*}
\gamma x = \bigg(\frac{\partial_{x} \cV(t,x,q)}E \frac{q^{kb}}{x^{kb}}\bigg)^{1/(p-1)}, \quad x \geq q. 
\end{equation*} 

The point is then to write the equation satisfied by $\cV$, namely the equivalent of 
\eqref{eq:full master PDE:MFG:1} but satisfied by $\cV$ instead of $V$. 
First, we observe that, in \eqref{eq:full master PDE:MFG:1}, $\sigma(x) \equiv  0$. 
Obviously, the difficult point is to rewrite $A_{\mu}$ and $A_{x\mu}$
as differential operators acting on the variables $q$ and $(x,q)$ respectively. 

A natural solution is to redo the computations used for deriving \eqref{eq:full master PDE:MFG:1} 
by replacing It\^o's formula
for the measures $(\hat{\mu}_{t})_{0 \leq t \leq T}$ 
by It\^o's formula for $(\hat{q}_{t})_{0 \leq t \leq T}$, taking benefit that 
$(\hat{q}_{t})_{0 \leq t \leq T}$ solves the SDE
\begin{equation}
\label{SDE:q}
d\hat{q}_t=\gamma \hat{q}_t dt + \sigma \hat{q}_t dW_t,
\end{equation}
which is a consequence of \eqref{fo:alpha_linear}
and \eqref{fo:diff}. 
Then the term $A_{\mu} \tilde{V}$ in \eqref{eq:full master PDE:MFG:1}, which reads
as the It\^o expansion of $V$ along $(\hat{\mu}_{t})_{0 \leq t \leq T}$, turns into 
the second-order differential operator associated to the SDE satisfied by $\hat{q}_{t}$, namely
\begin{equation*}
A_{q} \cV(t,x,q) = \gamma q \partial_{q} \cV(t,x,q) + \frac{1}{2} \sigma^2 q^2 \partial_{q}^2 \cV(t,x,q).
\end{equation*}
Similarly, the term $A_{x\mu} \tilde{V}$ in \eqref{eq:full master PDE:MFG:1}, which reads
as the bracket of the components in $\RR^d$ and in ${\mathcal P}_{2}(\RR^d)$ 
in the It\^o expansion, turns into 
the second-order differential operator associated to bracket of the SDEs
satisfied by $(X_{t})_{0 \leq t \leq T}$ in \eqref{SDE:pareto} and by $(\hat{q}_{t})_{0 \leq t \leq T}$, namely
\begin{equation*}
A_{xq} \cV(t,x,q) = \sigma^2 xq \partial_{xq}^2\cV(t,x,q).
\end{equation*}
Rewriting \eqref{eq:full master PDE:MFG:1}, we get 
\begin{equation}
\label{eq:full master PDE:pareto:1}
\begin{split}
&\partial_{t} \cV(t,x,q) + 
\frac{p-1}{p} E^{-1/(p-1)} \bigl( \partial_{x} \cV(t,x,q) \bigr)^{p/(p-1)} 
\Bigl( \frac{q^{kb/(p-1)}}{x^{kb/(p-1)}} \wedge 1 \Bigr) 
+ c \frac{x^{a+(k+1)b}}{k^bq^{kb}} {\mathbf 1}_{\{x \geq q\}}
\\
 &\hspace{5pt}
 +
\gamma q \partial_{q} \cV(t,x,q) + 
\frac12 \sigma^2 \bigl[ x^2 \partial_{x}^2 \cV(t,x,q) + q^2 
\partial_{q}^2 \cV(t,x,q) + 2 xq \partial_{xv}^2 \cV(t,x,q) \bigr]=0.
\end{split}
\end{equation}
Now we look for a constant $B >0$ such that
\begin{equation}
\label{eq:V:candidat}
\cV(t,x,q) = \cV(x,q) =B \frac{x^{p+bk}}{q^{bk}},
\end{equation}
solves the parameterized master equation \eqref{eq:full master PDE:pareto:1} on the set
$\{x \geq q\}$. Under the additional condition that $a+b=p$, $B$ must be the solution of the equation
\begin{equation*}
\frac{p-1}{p} E^{-1/(p-1)} \bigl( B (p+bk)
\bigr)^{p/(p-1)} +  \frac{c}{k^b} 
- \gamma B bk + 
\frac{\sigma^2}{2} B p(p-1)=0.
\end{equation*}
The condition \eqref{Cond:Pareto} reads
\begin{equation*}
\gamma = \Bigl( \frac{B(p+bk)}{E} \Bigr)^{1/(p-1)},
\end{equation*}
so that the above equation for $B$ becomes
\begin{equation*}
(p+bk)^{1/(p-1)} E^{-1/(p-1)}
\bigl( p-1 - \frac{bk}{p} \bigr)
B^{p/(p-1)} + \frac{\sigma^2}{2} p(p-1) B + \frac{c}{k^b}=0.
\end{equation*}
which always admits a solution if $p(p-1)<bk$.
The fact that 
\eqref{eq:full master PDE:pareto:1} is satisfied for $x \geq q$
is enough to prove that 
\begin{equation*}
\biggl( 
\cV(\hat{X}_{t},\hat{q}_{t})
+ \int_{0}^t f\bigl(\hat{X}_{s},\hat{\mu}_{s},\gamma \hat{X}_{s}
\bigr) ds 
\biggr)_{0 \leq t \leq T}, \quad \textrm{with} \ \hat{\mu}_{s} = \mu^{(\hat{q}_{s})}
\ {\rm for}\quad s \in [0,T],  
\end{equation*}
is a martingale, whenever
\begin{equation*}
d \hat{X}_{t} = \gamma \hat{X}_{t} dt + \sigma \hat{X}_{t}
dW_{t}^0, \quad t \in [0,T],
\end{equation*}
with $\hat{X}_{0} \sim \mu^{\hat{q}_{0}}$, and $(\hat{q}_{t})_{0 \leq t \leq T}$
also solves \eqref{SDE:q}. The reason is that $\hat{X}_{t} > \hat{q}_{t}$
for any $t \in [0,T]$ (equality $\hat{X}_{t} = \hat{q}_{t}$ holds along scenarios for
which $\hat{X}_{0}=\hat{q}_{0}$, which are of zero probability). 

The martingale property is a part of the verification Proposition 
\ref{pr:mfg_decoupling} for proving the optimality of $(\hat{X}_{t})_{0 \leq t \leq T}$
when $(\hat{\mu}_{t})_{0 \leq t \leq T}$ is the flow of conditional measures, but this is not sufficient. 
We must evaluate 
$\cV$ along a pair $(X_{t},\hat{q}_{t})_{0 \leq t \leq T}$, $(X_{t})_{0 \leq t \leq T}$
denoting a general controlled process satisfying \eqref{SDE:pareto}. Unfortunately, things then become more
difficult as $X_{t}$ might not be larger than $\hat{q}_{t}$. In other words, we are facing the fact 
that $\cV$ satisfies the PDE
\eqref{eq:full master PDE:pareto:1} on the set $\{x \geq q\}$ only. 
In order to circumvent this problem, a strategy consists in replacing $\cV$ by 
\begin{equation*}
\cV(x,q) = B x^p \Bigl( \frac{x^{bk}}{q^{bk}} \wedge 1 \Bigr),
\end{equation*}
for the same constant $B$ as above. Obviously, the PDE
\eqref{eq:full master PDE:pareto:1} is not satisfied when $x <q$, but
$\cV$ defines a subsolution on the set $\{0\leq x < q\}$, as \eqref{eq:full master PDE:pareto:1} holds but
with $=0$ replaced by $\geq 0$. 
Heuristically, this should show that 
\begin{equation}
\label{eq:pareto:martingale}
\biggl( 
\cV(X_{t},\hat{q}_{t})
+ \int_{0}^t f\bigl(X_{s},\hat{\mu}_{s},\alpha_{s}
\bigr) ds 
\biggr)_{0 \leq t \leq T}  
\end{equation}
is a submartingale when $(X_{t})_{0 \leq t \leq T}$ is an arbitrary controlled process driven by the control $(\alpha_{t})_{0 \leq t \leq T}$. 
Still, the justification requires some precaution as the function $\cV$ is not ${\mathcal C}^2$ (which is the standard framework
to apply It\^o's expansion),
its first-order derivatives being discontinuous on the diagonal $\{x=q\}$.
The argument for justifying the It\^o expansion is a bit technical so that we just give a sketchy proof of it. Basically,
we can write $\cV(X_{t},\hat{q}_{t}) = B (X_{t})^p [\varphi(X_{t}/\hat{q}_{t})]^{bk}$, 
with $\varphi(r) = \min(1,r)$. The key point is that 
$(X_{t}/\hat{q}_{t})_{0 \leq t \leq T}$ is always a bounded variation process, so that 
the expansion of $(\phi(X_{t}/\hat{q}_{t}))_{0 \leq t \leq T}$, for some function 
$\phi$, only requires to control $\phi'$ and not $\phi''$. 
Then, we can regularize $\varphi$ by a sequence $(\varphi_{n})_{n \geq 1}$
such that $(\varphi_{n})'(r) = 0$, for $r \leq 1-1/n$, 
$(\varphi_{n})'(r) = 1$, for $r \geq 1$ and  
$(\varphi_{n})'(r) \in [0,1]$ for $r \in [1-1/n,1]$. The fact that 
$(\varphi_{n})'(r)$ is uniformly bounded in $n$ permits 
to expand $(B (X_{t})^p [\varphi_{n}(X_{t}/\hat{q}_{t})]^{bk})_{0 \leq t \leq T}$
and then to pass to the limit. 

The submartingale property shows that 
\begin{equation}
\label{eq:pareto:submartingale}
\int_{\RR^d} 
\cV(x,\hat{q}_{0}) d\mu^{\hat{q}_{0}}(x) 
\leq \inf_{(\alpha_{t})_{0 \leq t \leq T}}
\biggl[ \int_{0}^T f(X_{t},\hat{q}_{t},\alpha_{t}) dt + 
\cV(X_{T},\hat{q}_{T}) \biggr],
\end{equation}
which, together with the martingale property along $(\hat{X}_{t})_{0 \leq t \leq T}$, shows
that equality holds and that the Pareto distributions $(\hat{\mu}_{t})_{0 \leq t \leq T}$
form a MFG equilibrium, provided $g$ is chosen as $\cV$. 
This constraint on the choice of $g$ can be circumvented by choosing 
$T=\infty$, as done in \cite{GueantLasryLions.pplnm}, in which case 
$f$ must be replaced by $e^{-rt}f$ for some discount rate $r>0$. 

The analysis in the case $T=\infty$ can be done in the following way.
In the proof of the martingale and submartingale properties, $\cV$ must replaced by $e^{-rt} \cV$. 
Plugging $e^{-rt} \cV$ and $e^{-rt} f$
in \eqref{eq:full master PDE:pareto:1} instead of $\cV$ and $f$, we understand that $\cV$ must now satisfy 
\eqref{eq:full master PDE:pareto:1} but with an additional $-r\cV$ in the left-hand side. Then, we can repeat 
the previous argument in order to identify the value of $B$ in \eqref{eq:V:candidat}. Finally, if 
$r$ is large enough, ${\mathbb E}[e^{-rT} \cV(\hat{X}_{T},\hat{q}_{T})]$ tends to $0$
as $T$ tends to the infinity in the martingale property 
\eqref{eq:pareto:martingale}. Similarly, if we restrict ourselves to a class of feedback controls
with a suitable growth, 
${\mathbb E}[e^{-rT} \cV(X_{T},\hat{q}_{T})]$ tends to $0$ in \eqref{eq:pareto:submartingale}, which permits to 
conclude. 

\subsection{Control of McKean-Vlasov Equations}
A similar framework could be used for considering the control of McKean-Vlasov equations. 
The analog of the strategy exposed in the previous paragraph would consist in limiting the optimization procedure
to controlled processes in \eqref{SDE:pareto} driven by controls $(\alpha_{t})_{0 \leq t \leq T}$
of the form $(\alpha_{t} = \gamma_{t} X_{t})_{0 \leq t \leq T}$ for some deterministic $(\gamma_{t})_{0 \leq t \leq T}$. 
Using an obvious extension of \eqref{fo:diff}, this would force 
the conditional marginal distributions of 
$(X_{t})_{0 \leq t \leq T}$ to be Pareto distributed. Exactly as above, this would transform the problem into 
a finite dimensional problem. Precisely, this would transform the problem into a finite dimensional optimal control problem. In that perspective, the corresponding master equation could be reformulated as an HJB equation in finite dimension. In comparison with, 
we emphasize, once again, that the master equation \eqref{eq:full master PDE:pareto:1} for the mean field game is not a HJB equation.

\section{Appendix: A Generalized Form of It\^o's Formula}
Our derivation of the master equation requires the use of a form of It\^o formula in a space of probability measures. This subsection is devoted to the proof of such a formula.

\subsection{Notion of Differentiability}
In Section \ref{se:ME}, we alluded to a specific notion of differentiability for functions of probability measures. 
The choice of this notion is dictated by the fact that 1) the probability measures we are dealing with appear as laws of random variables; 2) in trying to differentiate functions of measures, the infinitesimal variations which we consider  are naturally expressed as infinitesimal variations in the linear space of those random variables. The relevance of this notion of differentiability was argued by P.L. Lions in his lectures at the \emph{Coll\`ege de France} \cite{Lions}.  The notes \cite{Cardaliaguet} offer a readable account, and \cite{CarmonaDelarue_ap} provides several properties involving empirical measures. It is based on the \emph{lifting} of functions $\cP_2(\RR^d)\ni\mu\mapsto H(\mu)$ 
into functions $\t H$  defined on the Hilbert space $L^2(\t\Omega;\RR^d)$ over some probability space $(\t\Omega,\t\cF,\t\PP)$ by setting $\t H(\t X)=H({\mathcal L}(\t X))$, for $\t X \in L^2(\t \Omega;\RR^d)$, $\tilde{\Omega}$ being a Polish space and $\t \PP$ an atomless measure. 

Then, a function $H$
is said to be differentiable  at $\mu_0\in\cP_2(\RR^d)$ if there exists a random variable $\t X_0$ with law $\mu_0$, in other words 
satisfying ${\mathcal L}(\t X_0)=\mu_0$, such that the lifted function $\t H$ is Fr\'echet differentiable at $\t X_0$. 
Whenever this is the case, the Fr\'echet derivative of $\t H$ at $\t X_0$ can be viewed as an element of $L^2(\t\Omega;\RR^d)$ by identifying $L^2(\t\Omega;\RR^d)$ and its dual. It turns out that its distribution depends only upon the law $\mu_0$ and not upon the particular random variable $\t X_0$ having distribution $\mu_0$. See Section 6 in \cite{Cardaliaguet} for details. This Fr\'echet derivative $[D\t H](\t X_0)$ is called the representation of the derivative of $H$ at $\mu_0$ along the variable $\t X_{0}$.  It is shown in \cite{Cardaliaguet} that, as a random variable, it is of the form $\t h(\t X_0)$ for some deterministic measurable function $\t h : \RR^d \rightarrow \RR^d$, which is uniquely defined $\mu_0$-almost everywhere on $\RR^d$. The equivalence class of 
$\t h$ in $L^2(\RR^d,\mu_{0})$ being uniquely defined, it can be denoted by $\partial_{\mu} H(\mu_{0})$ (or $\partial
 H(\mu_{0})$ when no confusion is possible). It is then natural to call $\partial_\mu H(\mu_0)$ the derivative of $H$ at $\mu_{0}$
 and to identify it with a function $\partial_{\mu} H(\mu_{0})( \, \cdot \, ) : \RR^d \ni x \mapsto \partial_{\mu} H(\mu_{0})(x)
\in \RR^d$. 

This procedure permits to express $[D \t H](\t X_0)$ as a function of any random variable $\t X_{0}$ with distribution $\mu_0$, irrespective of where this random variable is defined.

\begin{remark}
\label{re:notation_der}
Since it is customary to identify a Hilbert space to its dual, we will identify $L^2(\tilde\Omega)$ with its dual, and in so doing, any derivative $D\t H (\t X)$ will be viewed as an element of $L^2(\t \Omega)$. In this way, the derivative in the direction $\t Y$ will be given by the inner product $[D\t H (\t X)]\cdot \t Y$. Accordingly, the second Frechet derivative $D^2\t H(\t X)$ which should be a linear operator from $L^2(\t \Omega)$ into itself because of the identification with its dual, will be viewed as a bilinear form on $L^2(\t \Omega)$. In particular, we shall use the notation $D^2\t H (\t X)[\t Y\cdot \t Z]$ for $\big([D^2\t H (\t X)] (\t Y)\big)\cdot \t Z$.

\end{remark}

\begin{remark}
\label{re:first_der}
The following result (see \cite{CarmonaDelarue_ap} for a proof) gives, though under stronger regularity assumptions on the Fr\'echet derivatives, a convenient way to handle this notion of differentiation with respect to probability distributions. If the function $\t H$ is Fr\'echet differentiable and if its Fr\'echet derivative is uniformly Lipschitz (i.e. there exists a constant $c>0$ such that $\| D\t H(\t X) - D\t H(\t X')\|\le c |\t X -\t X'|$ for all $\t X, \t X'$ in $L^2(\t \Omega)$), then there exists a function $\partial_\mu H$
$$
\cP_2(\RR^d)\times \RR^d \ni (\mu, x)\mapsto \partial_\mu H (\mu)(x)
$$
such as $|\partial_\mu H (\mu)(x)-\partial_\mu H (\mu)(x')|\le c|x-x'|$ for all $x,x'\in\RR^d$ and $\mu\in\cP_2(\RR^d)$, and for every $\mu\in\cP_2(\RR^d)$, $\partial_\mu H(\mu)(\t X)=D\t H(\t X)$ almost su
 if $\mu={\mathcal L}(\tilde{X})$.
\end{remark}

\subsection{It\^o's Formula along a Flow of Conditional Measures}
In the derivation of the master equation, the value function is expanded along a flow of conditional measures. As already explained in Subsection \ref{subse:derivation}, this requires a suitable construction of the lifting. 

Throughout this section, we assume that $(\Omega,{\mathcal F},\PP)$ is of the form 
$(\Omega^{0} \times \Omega^1,{\mathcal F}^0 \otimes {\mathcal F}^1,\PP^0 \otimes \PP^1)$, 
$(\Omega^0,{\mathcal F}^0,\PP^0)$ supporting the common noise $W^0$,
and $(\Omega^1,{\mathcal F}^1,\PP^1)$ the idiosyncratic noise $W$. 
So an element $\omega \in \Omega$ can be written as $\omega=(\omega^0,\omega^1)
\in \Omega^0 \times \Omega^1$,
and functionals $H(\mu(\omega^0))$ of a random probability 
measure $\mu(\omega^0) \in {\mathcal P}_{2}(\RR^d)$ with $\omega^0 \in \Omega^0$, can be lifted into 
$\tilde{H}(\tilde{X}(\omega^0,\cdot))=H({\mathcal L}(\tilde{X}(\omega^0,\cdot)))$, where $\tilde{X}(\omega^0,\cdot)$
is an element of $L^2(\tilde{\Omega}^1,\tilde{\mathcal F}^1,
\PP^1;\RR^d)$ with $\mu(\omega^0)$ as distribution,  
$(\tilde{\Omega}^1,\tilde{\mathcal F}^1,
\tilde\PP^1)$ being Polish and atomless. Put it differently, 
the random variable $\t X$ is defined on $(\t \Omega = \Omega^0 \times \t \Omega^1,
\t{\mathcal F}={\mathcal F}^0 \otimes \t {\mathcal F}^1,\t\PP= \PP^0 \otimes \t \PP^1)$. 

The objective is then to expand $(\tilde{H}(\tilde{\chi}_{t}(\omega^0,\cdot)))_{0 \leq t \leq T}$, 
where $(\tilde{\chi}_{t})_{0 \leq t \leq T}$ is the copy so constructed, of an It\^o process on $(\Omega,{\mathcal F},\PP)$ of the form:
\begin{equation*}
\chi_{t} = \chi_{0} + \int_{0}^t \beta_{s} ds + \int_{0}^t \int_{\Xi} \varsigma_{s,\xi}^0 W^0(d\xi,ds) + 
\int_{0}^t \varsigma_{s} dW_{s},
\end{equation*}
for $t \in [0,T]$,
assuming that the processes  $(\beta_{t})_{0 \leq t \leq T}$, $(\varsigma_{t})_{0 \leq t \leq T}$ 
and $(\varsigma_{t,\xi}^0)_{0 \leq t \leq T,\xi \in \Xi}$ are progressively measurable 
with respect to the filtration generated by $W$ and $W^0$ and square integrable, in the sense that 
\begin{equation}
\label{eq:ito:L2}
{\mathbb E} \int_{0}^T \biggl( 
\vert \beta_{t} \vert^2 +  \vert \varsigma_{t} \vert^2 + \int_{\Xi} \vert \varsigma_{t,\xi}^0 \vert^2 
d \nu(\xi) \biggr) dt < + \infty. 
\end{equation}
Denoting by $(\t W_{t})_{0 \leq t \leq T}$, $(\t \beta_{t})_{0 \leq t \leq T}$, $(\t \varsigma_{t})_{0 \leq t \leq T}$ 
and $(\t \varsigma_{t,\xi}^0)_{0 \leq t \leq T,\xi \in \Xi}$ the copies 
of $(W_{t})_{0 \leq t \leq T}$, $(\beta_{t})_{0 \leq t \leq T}$, $(\varsigma_{t})_{0 \leq t \leq T}$ 
and $(\varsigma_{t,\xi}^0)_{0 \leq t \leq T,\xi \in \Xi}$, we then have
\begin{equation*}
\t \chi_{t} = \t \chi_{0} + \int_{0}^t \t \beta_{s} ds + \int_{0}^t \int_{\Xi} \t \varsigma_{s,\xi}^0 W^0(d\xi,ds) + 
\int_{0}^t \t \varsigma_{s} d\t W_{s},
\end{equation*}
for $t \in [0,T]$.
In this framework, 
we emphasize that it makes sense to look at $\tilde{H}(\tilde{\chi}_{t}(\omega^0,\cdot))$, for $t \in [0,T]$, since 
\begin{equation*}
{\mathbb E}^0 \t {\mathbb E}^1 \bigl[ \sup_{0 \leq t \leq T} \vert \t \chi_{t} \vert^2 \bigr]
=
{\mathbb E}^0 {\mathbb E}^1 \bigl[ \sup_{0 \leq t \leq T} \vert \chi_{t} \vert^2 \bigr] <+ \infty,
\end{equation*}
where ${\mathbb E}^0$, ${\mathbb E}^1$
and $\t \EE^1$ are the expectations associated to $\PP^0$, $\PP^1$ and $\t \PP^1$ respectively. 

In order to simplify notations, we let 
$\check{\chi}_{t}(\omega^0)=\tilde{\chi}_{t}(\omega^0,\cdot)$ for $t \in [0,T]$, 
so that $(\ch \chi_{t})_{0 \leq t \leq T}$ 
is $L^2(\t \Omega^1,\t {\mathcal F}^1,\t \PP^1;\RR^d)$-valued,
$\PP^0$ almost surely. Similarly, we let 
$\check{\beta}_{t}(\omega^0)=\tilde{\beta}_{t}(\omega^0,\cdot)$, 
$\check{\varsigma}_{t}(\omega^0)=\tilde{\varsigma}_{t}(\omega^0,\cdot)$
$\check{\varsigma}_{t,\xi}(\omega^0)=\tilde{\varsigma}_{t,\xi}(\omega^0,\cdot)$,
for $t \in [0,T]$ and $\xi \in \Xi$. 
We then claim

\begin{proposition}
\label{prop:ito:simple}
On the top of the assumption and notation introduced right above, 
assume that $\tilde{H}$ is twice continuously Fr\'echet differentiable. Then,
we have $\PP^0$ almost surely, for all $t \in [0,T]$,
\begin{equation}
\label{eq:ito:simple}
\begin{split}
\tilde{H}\bigl(\ch \chi_{t}\bigr) &= 
\tilde{H}\bigl(\ch \chi_{0} \bigr) + \int_{0}^t  D \tilde{H} \bigl(\ch \chi_{s}\bigr)\cdot \ch \beta_{s}  ds  
+ \int_{0}^t \int_{\Xi }  D \tilde{H}\bigl(\ch \chi_{s}\bigr) \cdot \ch \varsigma_{s,\xi}^0  \;W^0(d\xi,ds)
\\
&\hspace{15pt} + \frac12 \int_{0}^t  \biggl(  D^2 \tilde{H}(\t \chi_{s}) \bigl[ \ch \varsigma_{s} \t G,\ch \varsigma_{s} \t G \bigr]
+
\int_{\Xi} D^2 \tilde{H}\bigl(\ch \chi_{s}\bigr) \bigl[ 
\ch \varsigma_{s,\xi}^0,\ch \varsigma_{s,\xi}^0 \bigr]  d\nu(\xi)   \biggr) ds. 
\end{split}
\end{equation}
where $\t G$ is an ${\mathcal N}(0,1)$-distributed random variable on 
$(\t \Omega^1,\t {\mathcal F}^1, \t \PP^1)$, independent of $(\t W_{t})_{t \geq 0}$. 
\end{proposition}
\begin{remark}
\label{re:ito}
Following Remark \ref{re:first_der} above, one can specialize It\^o's formula to a situation with smoother derivatives. 
See \cite{ChassagneuxCrisanDelarue} for a more detailed account. 
Indeed, if one assumes that 
\begin{enumerate}\itemsep=-2pt
\item the function $H$ is $C^1$ in the sense given above and its first derivative is Lipschitz;
\item for each fixed $x\in\RR^d$, the function $\mu \mapsto \partial_\mu H(x,\mu)$ is differentiable with Lipschitz derivative, and consequently, there exists a function 
$$
(\mu,x',x)\mapsto \partial^2_{\mu}H(x,\mu)(x') \in \RR^{d \times d}
$$
which is Lipschitz in $x'$ uniformly with respect to $x$ and $\mu$ and such that 
$\partial^2_{\mu}H(x,\mu)(\t X)$ gives the Fr\'echet derivative 
of $\mu\mapsto \partial_\mu H(x,\mu)$ for every $x\in\RR^d$ as long as $\cL(\t X) = \mu$;
\item for each fixed $\mu\in\cP_2(\RR^d)$, the function $x \mapsto \partial_\mu H(x,\mu)$ is differentiable with Lipschitz derivative, and consequently, there exists a bounded function $(x,\mu)\mapsto \partial_x\partial_\mu H(x,\mu)
\in \RR^{d \times d}$ 
giving the value of its derivative.
\end{enumerate}
Then, the second order term appearing in It\^o's formula can be expressed as the sum of two explicit operators whose interpretations are more natural. Indeed, the second Fr\'echet derivative $D^2\t H(\t X)$ can be written as the linear operator $\t Y\mapsto A\t Y$ on $L^2(\t \Omega^1,\t {\mathcal F}^1,\PP^1;\RR^d)$ defined by
$$
[A\t Y](\t \omega^1)=\int_{\t \Omega^1}
\partial_{\mu}^2 H\bigl(\t X(\t \omega^1),{\mathcal L}(\t X)\bigr)\bigl(\t X'(\omega')\bigr)\t Y(\omega')\,d\t \PP^1(\omega')\;+\;\partial_x\partial_\mu H\bigl({\mathcal L}({\t X}),\t X(\omega^1)\bigr)\t Y(\omega^1).
$$
\end{remark}

The derivation of the master equation actually requires a more general result than Proposition \ref{prop:ito:simple}. 
Indeed one needs to expand $(\tilde{H}(X_{t},\ch{\chi}_{t}))_{0 \leq t \leq T}$
for a function $\tilde{H}$ of $(x,\tilde{X}) \in \RR^d \times L^2(\tilde{\Omega}^1,
\t {\mathcal F}^1, \t \PP^1;\RR^d)$. As before, 
$(\ch{\chi}_{t})_{0 \leq t \leq T}$ is understood as 
$(\tilde{\chi}_{t}(\omega^0,\cdot))_{0 \leq t \leq T}$. 
The process $(X_{t})_{0 \leq t \leq T}$ is assumed to be another It\^o process, 
defined on the original space $(\Omega,{\mathcal F},\PP)
= (\Omega^0 \times \Omega^1,{\mathcal F}^0 \otimes 
{\mathcal F}^1,\PP^0 \otimes \PP^1)$, with dynamics of the form
\begin{equation*}
X_{t} = X_{0} + \int_{0}^t b_{s} ds + \int_{0}^t \int_{\Xi} \sigma_{s,\xi}^0 W^0(d\xi,ds) + 
\int_{0}^t \sigma_{s} dW_{s},
\end{equation*}
for $t \in [0,T]$, the processes  $(b_{t})_{0 \leq t \leq T}$, $(\sigma_{t})_{0 \leq t \leq T}$ 
and $(\sigma_{t,\xi}^0)_{0 \leq t \leq T,\xi \in \Xi}$ being progressively-measurable 
with respect to the filtration generated by $W$ and $W^0$, and square integrable 
as in \eqref{eq:ito:L2}. 
Under these conditions, the result of Proposition 
\ref{prop:ito:simple} can be extended to: 

\begin{proposition}
\label{prop:ito:joint}
On the top of the above assumptions and notations, assume that $\tilde{H}$ is twice continuously Fr\'echet differentiable
on $\RR^d \times L^2(\t \Omega^1,\t {\mathcal F}^1,\t \PP^1;\RR^d)$. Then,
we have $\PP$ almost surely, for all $t \in [0,T]$,
\begin{equation*}
\begin{split}
&\tilde{H}
\bigl(X_{t},\ch \chi_{t}\bigr) = \tilde{H}\bigl(X_{0},\ch \chi_{0}\bigr) 
\\
&\hspace{5pt}
+ \int_{0}^t  \Bigl( \langle \partial_{x} \tilde{H}\bigl(X_{s},\ch \chi_{s}\bigr), b_{s}\rangle
+
D_{\mu} \tilde{H} \bigl(\ch \chi_{s}\bigr) \cdot \ch \beta_{s} \Bigr) ds
+ 
\int_{0}^t  \langle
\partial_{x} \tilde{H}\bigl(X_{s},\ch \chi_{s}\bigr), \sigma_{s}\rangle dW_{s}
\\
&\hspace{5pt} + \int_{0}^t \int_{\Xi } 
\Bigl( \langle \partial_{x} \tilde{H}\bigl(X_{s},\ch \chi_{s}\bigr), \sigma^0_{s,\xi}\rangle
+ 
D_{\mu} \tilde{H}\bigl(X_{s},\ch \chi_{s}\bigr)\cdot \ch \varsigma_{s,\xi}^0  \Bigr)\; W^0(d\xi,ds)
\\
&\hspace{5pt} + \frac12 \int_{0}^t \int_{\Xi}
\Bigl( {\rm trace} \bigl[ 
\partial^2_{x} \tilde{H}\bigl(X_{s},\ch \chi_{s}\bigr)
\sigma^0_{s,\xi}( \sigma^0_{s,\xi})^{\dagger}
\bigr]
+
D^2_{\mu} \tilde{H}\bigl(X_{s},\ch \chi_{s}\bigr)  \bigl[ 
\ch \varsigma_{s,\xi}^0,\ch \varsigma_{s,\xi}^0
\bigr] 
\Bigr) d\nu(\xi) ds
\\
&\hspace{5pt} 
+ \frac12 \int_{0}^t 
\bigg({\rm trace} \bigl[ 
\partial^2_{x} \tilde{H}\bigl(X_{s},\ch \chi_{s}\bigr)
\sigma_{s}( \sigma_{s})^{\dagger}
\bigr] + D^2_{\mu} \tilde{H}\bigl(X_{s},\ch \chi_{s}\bigr) \bigl[ \ch \varsigma_{s} \t G,\ch
\varsigma_{s} \t G  \bigr] \bigg)ds \\
&\hspace{5pt} 
+ 
\int_{0}^t 
%\bigg( \bigl\langle\partial_{x} D_{\mu}
%\tilde{H}\bigl(X_{s},\ch \chi_{s}\bigr) \cdot \ch
%\varsigma_{s}\, ,\, \sigma_{s}\bigr\rangle
%+
\int_{\Xi}\bigl\langle\partial_{x} D_{\mu}
\tilde{H}\bigl(X_{s},\ch \chi_{s}\bigr) \cdot \ch
\varsigma_{s,\xi}^0\, ,\, \sigma_{s,\xi}^0
\bigr\rangle  d\nu(\xi) ds. 
\end{split}
\end{equation*}
where $\t G$ is an ${\mathcal N}(0,1)$-distributed random variable on 
$(\t \Omega^1,\t {\mathcal F}^1, \t \PP^1)$, independent of $(\t W_{t})_{t \geq 0}$. 
The partial derivatives in the infinite dimensional component are denoted with the index 
`$\mu$'. In that framework, the term 
$\langle  \partial_{x} D_{\mu}
\tilde{H}(X_{s},\ch \chi_{s}) \cdot \ch
\varsigma_{s,\xi}^0 ,\sigma_{s,\xi}^0
\rangle$
reads
$$ \sum_{i=1}^d \{ \partial_{x_{i}} D_{\mu}
\tilde{H}(X_{s},\ch \chi_{s}) \cdot \ch
\varsigma_{s,\xi}^0 \} \bigl( \sigma_{s,\xi}^0 \bigr)_{i}.$$
\end{proposition}

\subsection{Proof of It\^o's Formula} 
We only provide the proof of Proposition \ref{prop:ito:simple} as 
the proof of Proposition \ref{prop:ito:joint} is similar. 

\vskip 4pt
By a standard continuity argument, it is sufficient to prove that Equation 
\eqref{eq:ito:simple} holds for any $t \in [0,T]$ $\PP^0$-almost surely. 
In particular, we can choose $t=T$. 
Moreover, by a standard approximation argument, it is sufficient to consider the case of
simple processes $(\beta_{t})_{0 \leq t \leq T}$, $(\varsigma_{t})_{0 \leq t \leq T}$
and $(\varsigma_{t,\xi}^0)_{0 \leq t \leq T,\xi}$  of the form
\begin{equation*}
\beta_{t} = \sum_{i=0}^{M-1} \beta_{i} {\mathbf 1}_{[\tau_{i},\tau_{i+1})}(t), \quad
\varsigma_{t} = \sum_{i=0}^{M-1} \varsigma_{i} {\mathbf 1}_{[\tau_{i},\tau_{i+1})}(t), \quad 
\varsigma_{t,\xi}^0 = \sum_{i=0}^{M-1} \sum_{j=1}^N 
\varsigma^0_{i,j}
{\mathbf 1}_{[\tau_{i},\tau_{i+1})}(t)
{\mathbf 1}_{A_{j}}(\xi),
\end{equation*}
where $M,N \geq 1$, $0=\tau_{0}<\tau_{1} < \dots < \tau_{M}=T$, 
$(A_{j})_{1 \leq j \leq N}$ are piecewise disjoint Borel subsets of $\Xi$ and 
$(\beta^i,\varsigma^i,\varsigma^0_{i,j})_{1 \leq j \leq N}$ are bounded ${\mathcal F}_{\tau_{i}}$-measurable random variables. 

The strategy is taken from \cite{ChassagneuxCrisanDelarue} and consists in splitting 
$\tilde{H}(\ch \chi_{T}) - \tilde{H}(\ch \chi_{0})$ into
\begin{equation*}
\tilde{H}(\ch \chi_{T}) - \tilde{H}(\ch \chi_{0})
= \sum_{k=0}^{K-1} \bigl( 
\tilde{H}(\ch \chi_{t_{k+1}}) - \tilde{H}(\ch \chi_{t_{k}}) \bigr),
\end{equation*}
where $0=t_{0}< \dots < t_{K}=T$ is a subdivision of $[0,T]$ of step $h$ such that, for any $k \in \{0,\dots,K-1\}$,
there exists some $i \in \{0,\dots,M-1\}$ such that
$[t_{k},t_{k+1}) \subset [\tau_{i},\tau_{i+1})$. 
We then  start with approximating 
a general increment $\tilde{H}(\ch \chi_{t_{k+1}}) - \tilde{H}(\ch \chi_{t_{k}})$, omitting to specify the  
dependence upon $\omega^0$. By Taylor's formula, we know that
we can find some $\delta \in [0,1]$ such that  
\begin{equation}
\label{eq:22:1:1}
\begin{split}
&\tilde{H}(\ch \chi_{t_{k+1}}) - \tilde{H}(\ch \chi_{t_{k}}) 
\\
&= 
D \tilde{H}(\ch \chi_{t_{k}})  \cdot (\ch \chi_{t_{k+1}} - \ch \chi_{t_{k}}) 
+ \frac{1}{2}
D^2 \tilde{H}\bigl(\ch \chi_{t_{k}} + \delta (\ch \chi_{t_{k+1}}- \ch \chi_{t_{k}}) \bigr)  
\bigl( \ch \chi_{t_{k+1}} - \ch \chi_{t_{k}}, \ch \chi_{t_{k+1}} - \ch \chi_{t_{k}} \bigr)
\\
&= 
D \tilde{H}(\ch \chi_{t_{k}})  \cdot (\ch \chi_{t_{k+1}} - \ch \chi_{t_{k}}) + \frac{1}{2}
D^2 \tilde{H}(\ch \chi_{t_{k}}) 
\bigl( \ch \chi_{t_{k+1}} - \ch \chi_{t_{k}}, \ch \chi_{t_{k+1}} - \ch \chi_{t_{k}} \bigr) 
\\
&\hspace{15pt} + 
\bigl[
D^2 \tilde{H}\bigl(\ch \chi_{t_{k}} + \delta (\ch \chi_{t_{k+1}}- \ch \chi_{t_{k}}) \bigr) 
- D^2 \tilde{H}\bigl(\ch \chi_{t_{k}} \bigr) 
\bigr]
\bigl( \ch \chi_{t_{k+1}} - \ch \chi_{t_{k}}, \ch \chi_{t_{k+1}} - \ch \chi_{t_{k}} \bigr).
\end{split}
\end{equation}
By Kolmogorov continuity theorem, we know that, $\PP^0$ almost surely,
the mapping $[0,T] \ni t \mapsto \tilde{\chi}_{t} \in L^2(\t \Omega^1,\t {\mathcal F}^1, \t \PP^1;\RR^d)$
is continuous. Therefore, $\PP^0$ almost surely, the mapping 
$ (s,t,\delta) \mapsto D^2 \tilde{H}(\ch \chi_{t} + \delta (\ch \chi_{s}- \ch \chi_{t}))$
is continuous from $[0,T]^2 \times [0,1]$ to the space of bounded operators from 
$L^2(\t \Omega^1,\t {\mathcal F}^1, \t \PP^1;\RR^d)$ into itself, which proves that, $\PP^0$ almost surely, 
\begin{equation*}
\lim_{h \searrow 0} \sup_{s,t \in [0,T], \vert t-s \vert \leq h}
\sup_{\delta \in [0,1]}
\vvvert
D^2 \tilde{H}\bigl(\ch \chi_{t} + \delta (\ch \chi_{t+h}- \ch \chi_{t}) \bigr) 
- D^2 \tilde{H}\bigl(\ch \chi_{t} \bigr) 
\vvvert_{2,\tilde{\Omega}^1}=0,
\end{equation*} 
$\vvvert \cdot \vvvert_{2,\tilde{\Omega}^1}$ denoting the 
operator norm on the space of bounded operators on 
$L^2(\t \Omega^1,\t {\mathcal F}^1, \t \PP^1;\RR^d)$. Now,
\begin{equation*}
\begin{split}
&\biggl\vert \sum_{k=0}^{K-1}
\bigl[
D^2 \tilde{H}\bigl(\ch \chi_{t_{k}} + \delta (\ch \chi_{t_{k+1}}- \ch \chi_{t_{k}}) \bigr) 
- D^2 \tilde{H}\bigl(\ch \chi_{t_{k}} \bigr) 
\bigr]
\bigl( \ch \chi_{t_{k+1}} - \ch \chi_{t_{k}}, \ch \chi_{t_{k+1}} - \ch \chi_{t_{k}} \bigr)
\biggr\vert
\\
&\leq \sup_{s,t \in [0,T], \vert t-s \vert \leq h}
\sup_{\delta \in [0,1]}
\vvvert
D^2 \tilde{H}\bigl(\ch \chi_{t} + \delta (\ch \chi_{s}- \ch \chi_{t}) \bigr) 
- D^2 \tilde{H}\bigl(\ch \chi_{t} \bigr) 
\vvvert_{2,\tilde{\Omega}^1} \sum_{k=0}^{K-1}
\Vert
\ch \chi_{t_{k+1}} - \ch \chi_{t_{k}} \Vert_{L^2(\t \Omega)}^2.
\end{split}
\end{equation*}
Since 
\begin{equation*}
\EE^0 \biggl[ \sum_{k=0}^{K-1}
\Vert
\ch \chi_{t_{k+1}} - \ch \chi_{t_{k}} \Vert_{L^2(\t \Omega)}^2 \biggr]
\leq C \sum_{k=0}^{K-1} \bigl(t_{k+1}- t_{k}\bigr) \leq CT, 
\end{equation*}
we deduce that 
\begin{equation}
\label{eq:ito:proof:1}
\biggl\vert \sum_{k=0}^{K-1}
\bigl[
D^2 \tilde{H}\bigl(\ch \chi_{t_{k}} + \delta (\ch \chi_{t_{k+1}}- \ch \chi_{t_{k}}) \bigr) 
- D^2 \tilde{H}\bigl(\ch \chi_{t_{k}} \bigr) 
\bigr]
\cdot \bigl( \ch \chi_{t_{k+1}} - \ch \chi_{t_{k}}, \ch \chi_{t_{k+1}} - \ch \chi_{t_{k}} \bigr)
\biggr\vert \rightarrow 0
\end{equation}
in $\PP^0$ probability as $h$ tends to $0$. 
We now compute the various terms appearing in \eqref{eq:22:1:1}. We write
\begin{equation*}
\begin{split}
&D \tilde{H}(\ch \chi_{t_{k}})  \cdot (\ch \chi_{t_{k+1}} - \ch \chi_{t_{k}})
= D \tilde{H}(\ch \chi_{t_{k}})  \cdot \int_{t_{k}}^{t_{k+1}} \t \beta_{s}(
\omega^0,\cdot) ds
\\
&\hspace{5pt} + D \tilde{H}(\ch \chi_{t_{k}}) \cdot 
\biggl[
\biggl( \int_{t_{k}}^{t_{k+1}}
\int_{\Xi}
\t \varsigma_{s,\xi}^0 W^0(d\xi,ds) \biggr)(\omega^0,\cdot) \biggr] + 
D \tilde{H}(\ch \chi_{t_{k}}) \cdot \biggl[ 
\biggl( \int_{t_{k}}^{t_{k+1}} \t \varsigma_{s} d\t W_{s} \biggr) \biggr] (\omega^0,\cdot). 
\end{split}
\end{equation*}
Assume that, for some $0 \leq i \leq M-1$, $\tau_{i} \leq t_{k} < t_{k+1} \leq \tau_{i+1}$. Then,
\begin{equation}
\label{eq:ito:drift:1}
D \tilde{H}(\ch \chi_{t_{k}})  \cdot \int_{t_{k}}^{t_{k+1}} \t \beta_{s}(
\omega^0,\cdot) ds
= \bigl(t_{k+1}-t_{k} \bigr) D \tilde{H}(\ch \chi_{t_{k}}) \cdot \t \beta_{t_{k}}(\omega^0,\cdot).
\end{equation}
Note that the right-hand side is well-defined as $\beta_{t_{k}}$ is bounded. 
Similarly, we notice that 
\begin{equation*}
D \tilde{H}(\ch \chi_{t_{k}})  \cdot \biggl[
\biggl(
 \int_{t_{k}}^{t_{k+1}} \t \varsigma_{s} d \t W_{s} \biggr)(\omega^0,\cdot) \biggr] 
= \bigl(t_{k+1}-t_{k} \bigr) D \tilde{H}(\ch \chi_{t_{k}}) \cdot \bigl[ \t \varsigma_{t_{k}}(\omega^0,\cdot) 
\bigl(\t W_{t_{k+1}} - \t W_{t_{k}} \bigr) \bigr]. 
\end{equation*}
Now, using the specific form of $D \tilde{H}$,  
$D \tilde{H}(\ch \chi_{t_{k}}(\omega^0))=(\t \omega^1 
\mapsto \partial_{\mu}H(\t \chi_{t_{k}}(\omega^0,\t \omega^1)))$ appears to be a $\t {\mathcal F}_{t_{k}}$-measurable random variable,  
and as such, it is orthogonal to $\t \varsigma_{t_{k}}(\omega^0,\cdot) (\t W_{t_{k+1}} -\t W_{t_{k}})$, which shows that 
\begin{equation}
\label{eq:22:1:3}
D \tilde{H}(\ch \chi_{t_{k}})  \cdot \biggl[
\biggl(
 \int_{t_{k}}^{t_{k+1}} \t \varsigma_{s} d\t W_{s} \biggr)(\omega^0,\cdot) \biggr] 
=  0. 
\end{equation}
Finally, 
\begin{equation*}
\begin{split}
&D \tilde{H}(\ch \chi_{t_{k}})  \cdot \biggl[
\biggl(
 \int_{t_{k}}^{t_{k+1}} \int_{\Xi}
\t \varsigma^0_{s,\xi} W^0(d\xi,ds) \biggr)(\omega^0,\cdot) \biggr] 
\\
&=  D \tilde{H}(\ch \chi_{t_{k}})  \cdot \biggl[ \sum_{j=1}^N \t \varsigma^0_{i,j}(\omega^0,\cdot)
W^0\bigl(A_{j} \times [t_{k},t_{k+1}) \bigr)(\omega^0) \biggr].
\end{split}
\end{equation*}
Now, $W^0\bigl(A_{j} \times [t_{k},t_{k+1}) \bigr)(\omega^0)$ behaves as a constant in the linear form above. Therefore,
\begin{equation}
\label{eq:22:1:5}
\begin{split}
&D \tilde{H}(\ch \chi_{t_{k}})  \cdot \biggl[
\biggl(
 \int_{t_{k}}^{t_{k+1}} \int_{\Xi}
 \t \varsigma^0_{s,\xi} W^0(d\xi,ds) \biggr)(\omega^0,\cdot) \biggr] 
\\
&\hspace{15pt} =  \sum_{j=1}^N
D \tilde{H}(\ch \chi_{t_{k}})  \cdot \t \varsigma^0_{i,j}(\omega^0,\cdot)
W^0\bigl(A_{j} \times [t_{k},t_{k+1}) \bigr)(\omega^0)
\\
&\hspace{15pt} = \biggl[\int_{t_{k}}^{t_{k+1}} \int_{\Xi}
\bigl\{
D \tilde{H}(\ch \chi_{t_{k}}) \cdot \t \varsigma^0_{s,\xi}(\omega^0,\cdot) \bigr\}
W(d\xi,ds) \biggr](\omega^0).
\end{split} 
\end{equation}
Therefore, in analogy with \eqref{eq:ito:proof:1}, we deduce 
from \eqref{eq:ito:drift:1}, \eqref{eq:22:1:3}
and \eqref{eq:22:1:5} that
\begin{equation*}
\sum_{k=0}^{K-1}
D \tilde{H}(\ch \chi_{t_{k}})  \cdot (\ch \chi_{t_{k+1}} - \ch \chi_{t_{k}})
\rightarrow \int_{0}^T D \tilde{H}(\tilde{X}_{s}) \cdot \ch \beta_{s} ds 
+ \int_{0}^T \int_{\Xi} \bigl\{
D \tilde{H}(\ch \chi_{s}) \cdot \ch \varsigma^0_{s,\xi} \bigr\}
W(d\xi,ds),
\end{equation*} 
in $\PP^0$ probability as $h$ tends to $0$. 
\vskip 4pt
We now reproduce this analysis for the second order derivatives. 
We need to compute: 
\begin{equation*}
\begin{split}
\Gamma_k &:= D^2 \tilde{H}(\ch \chi_{t_{k}}) \Bigl[\t \beta_{t_{k}}(\omega^0,\cdot) \bigl(t_{k+1}-t_{k}\bigr) + 
\t \varsigma_{t_{k}}(\omega^0,\cdot) \bigl( \t W_{t_{k+1}}- \t W_{t_{k}}\bigr)
\\
&\hspace{200pt}
+ \sum_{j=1}^N \t \varsigma_{i,j}^0(\omega^0,\cdot) W^0\bigl([t_{k},t_{k+1}) \times A_{j}\bigr)(\omega^0), 
\\
&\hspace{10pt}
\t \beta_{t_{k}}(\omega^1,\cdot) \bigl(t_{k+1}-t_{k}\bigr) + 
\t \varsigma_{t_{k}}(\omega^0,\cdot) \bigl( \t W_{t_{k+1}}-\t  W_{t_{k}}\bigr)
+ \sum_{j=1}^N \t \varsigma_{i,j}^0(\omega^0,\cdot) W^0\bigl([t_{k},t_{k+1}) \times A_{j}\bigr)(\omega^0)
\Bigr]. 
\end{split}
\end{equation*}
Clearly, the drift has very low influence on the value of $\Gamma_{k}$. 
Precisely, for investigating the limit (in $\PP^0$ probability)
of $\sum_{k=0}^{K-1} \Gamma_{k}$, we can focus on the `reduced' version of $\Gamma_{k}$:
\begin{equation*}
\begin{split}
\Gamma_{k} &:= D^2 \tilde{H}(\ch \chi_{t_{k}})  \Bigl[ 
\t \varsigma_{t_{k}}(\omega^0,\cdot) \bigl( \t W_{t_{k+1}}- \t W_{t_{k}}\bigr)
+ \sum_{j=1}^N \varsigma_{i,j}^0(\omega^0,\cdot) W^0\bigl([t_{k},t_{k+1}) \times A_{j}\bigr)(\omega^0), 
\\
&\hspace{50pt}
\t \varsigma_{t_{k}}(\omega^0,\cdot) \bigl( \t W_{t_{k+1}}- \t W_{t_{k}}\bigr)
+ \sum_{j=1}^N \varsigma_{i,j}^0(\omega^0,\cdot) W^0\bigl([t,t+h] \times A_{j}\bigr)(\omega^0)
\Bigr]. 
\end{split}
\end{equation*}
We first notice that
\begin{equation*}
D^2 \tilde{H}(\ch \chi_{t_{k}})  \bigl[\t \varsigma_{t_{k}}(\omega^0,\cdot) \bigl( 
\t W_{t_{k+1}} - \t W_{t_{k}} \bigr) ,
\t \varsigma_{i,j}^0(\omega^0,\cdot) W^0\bigl([t_{k},t_{k+1}) \times A_{j}\bigr)(\omega^0) \bigr]= 0,
\end{equation*}
the reason being that 
\begin{equation*}
\begin{split}
&D^2 \tilde{H}(\ch \chi_{t_{k}})   \bigl[\t \varsigma_{t_{k}}(\omega^0,\cdot) \bigl( 
\t W_{t_{k+1}} - \t W_{t_{k}} \bigr) ,
\t \varsigma_{i,j}^0(\omega^0,\cdot) W^0\bigl([t_{k},t_{k+1}) \times A_{j}\bigr)(\omega^0) \bigr]
\\
&= \lim_{\epsilon \rightarrow 0} \epsilon^{-1}\bigl[ D \tilde{H}\bigl(
\ch \chi_{t_{k}} + \epsilon \t \varsigma_{i,j}^0(\omega^0,\cdot) W^0\bigl([t_{k},t_{k+1}) \times A_{j}\bigr)(\omega^0)
\bigr) 
\\
&\hspace{150pt}
-  
D \tilde{H}(
\ch \chi_{t_{k}}) \bigr]  \bigl[Ê\t \varsigma_{t_{k}}(\omega^0,\cdot) \bigl(
\t W_{t_{k+1}}- \t W_{t_{k}} \bigr) \bigr],
\end{split}
\end{equation*}
which is zero by the independence argument used in \eqref{eq:22:1:3}. 
Following the proof of \eqref{eq:22:1:5}, 
\begin{equation*}
\begin{split}
&D^2 \tilde{H}(\ch \chi_{t_{k}})  
\Bigl[\sum_{j=1}^N \t \varsigma_{i,j}^0(\omega^0,\cdot) W^0\bigl([t_{k},t_{k+1}) \times A_{j}\bigr)(\omega^0) ,
\sum_{j=1}^N \t \varsigma_{i,j}^0(\omega^0,\cdot) W^0\bigl([t_{k},t_{k+1}) \times A_{j}\bigr)(\omega^0) 
\Bigr]
\\
&= 
\sum_{j,j'=1}^N D^2 \tilde{H}(\ch \chi_{t_{k}})  \bigl[ 
\t \varsigma_{i,j}^0(\omega^0,\cdot),\t \varsigma_{i,j'}^0(\omega^0,\cdot)
\bigr] 
W^0\bigl([t_{k},t_{k+1}) \times A_{j}\bigr)(\omega^0)
W^0\bigl([t_{k},t_{k+1}) \times A_{j'}\bigr)(\omega^0). 
\end{split}
\end{equation*}
The second line reads as a the bracket of a discrete stochastic integral. 
Letting $\ch \varsigma_{i,j}^0(\omega^0)
= \t \varsigma_{i,j}^0(\omega^0,\cdot)$,
it is quite standard to check 
\begin{equation*}
\begin{split}
&\sum_{k=0}^{K-1}
\sum_{j,j'=1}^N D^2 \tilde{H}(\ch \chi_{t_{k}})  \bigl[ 
\ch \varsigma_{i,j}^0,\ch \varsigma_{i,j'}^0
\bigr] 
W^0\bigl([t_{k},t_{k+1}) \times A_{j}\bigr)
W^0\bigl([t_{k},t_{k+1}) \times A_{j'}\bigr) 
\\
&\hspace{15pt}
- \sum_{k=0}^{K-1}
\sum_{j=1}^N D^2 \tilde{H}(\ch \chi_{t_{k}})  \bigl[ 
\ch \varsigma_{i,j}^0,\ch \varsigma_{i,j}^0
\bigr] \bigl(t_{k+1} - t_{k}\bigr) \nu(A_{j}) \rightarrow 0
\end{split}
\end{equation*}
in $\PP^0$ probability as $h$ tends to $0$. Noticing that 
\begin{equation*}
\sum_{k=0}^{K-1}
\sum_{j=1}^N D^2 \tilde{H}(\ch \chi_{t_{k}})  \bigl[ 
\ch \varsigma_{i,j}^0,\ch \varsigma_{i,j}^0
\bigr] \bigl(t_{k+1} - t_{k}\bigr) \nu(A_{j})
= \sum_{k=0}^{K-1}
\int_{t_{k}}^{t_{k+1}} \int_{\Xi}
D^2 \tilde{H}(\ch \chi_{t_{k}})  \bigl[ 
\ch \varsigma_{s,\xi}^0,\ch \varsigma_{s,\xi}^0
\bigr]  d\nu(\xi) ds,
\end{equation*}
we deduce that 
\begin{equation*}
\begin{split}
&\sum_{k=0}^{K-1}
\sum_{j,j'=1}^N D^2 \tilde{H}(\ch \chi_{t_{k}})  \bigl[ 
\ch \varsigma_{i,j}^0,\ch \varsigma_{i,j'}^0
\bigr] 
W^0\bigl([t_{k},t_{k+1}) \times A_{j}\bigr)
W^0\bigl([t_{k},t_{k+1}) \times A_{j'}\bigr)
\\
&\hspace{15pt}
- \int_{0}^T \int_{\Xi}
D^2 \tilde{H}(\ch \chi_{s})  \bigl[ 
\ch \varsigma_{s,\xi}^0,\ch \varsigma_{s,\xi}^0
\bigr]  d\nu(\xi) ds
 \rightarrow 0
\end{split}
\end{equation*}
in $\PP^0$ probability as $h$ tends to $0$.
It remains to compute
\begin{equation*}
D^2 \tilde{H}(\ch \chi_{t_{k}})  \bigl[\t  \varsigma_{t_{k}}(\omega^0,\cdot) \bigl( 
\t W_{t_{k+1}} - \t W_{t_{k}} \bigr) , \t \varsigma_{t_{k}}(\omega^0,\cdot) \bigl(
\t W_{t_{k+1}} - \t W_{t_{k}} \bigr) \bigr]. 
\end{equation*}
Recall that this is the limit
\begin{equation*}
\begin{split}
\lim_{\varepsilon \rightarrow 0} \frac{1}{\varepsilon^2}
&\bigl[ \tilde{H}\bigl(\t \chi_{t_{k}}(\omega^0,\cdot) + \varepsilon 
\t \varsigma_{t_{k}}(\omega^0,\cdot) (\t W_{t_{k+1}}- \t W_{t_{k}}) \bigr)
\\
&\hspace{15pt}+ \tilde{H}\bigl(\t \chi_{t_{k}}(\omega^0,\cdot) - \varepsilon 
\t \varsigma_{t_{k}}(\omega^0,\cdot) (\t W_{t_{k+1}}- \t W_{t_{k}}) \bigr)
- 2 \tilde{H}\bigl(\t \chi_{t_{k}}(\omega^0,\cdot) \bigr) \bigr],
\end{split}
\end{equation*}
which is the same as 
\begin{equation*}
\lim_{\varepsilon \rightarrow 0} \frac{1}{\varepsilon^2}
\bigl[ \tilde{H}\bigl(\t \chi_{t_{k}}(\omega^0,\cdot) + \varepsilon 
\t \varsigma_{t_{k}}(\omega^0,\cdot) \sqrt{t_{k+1}-t_{k}} \t G \bigr)
- \tilde{H}\bigl(\t \chi_{t_{k}}(\omega^0,\cdot) \bigr) \bigr],
\end{equation*}
where $\t G$ is independent of $(\t W_{t})_{0 \leq t \leq T}$, and ${\mathcal N}(0,1)$ distributed. Therefore, 
\begin{equation*}
D^2 \tilde{H}(\ch \chi_{t_{k}})  \bigl[ \ch \varsigma_{t_{k}} \bigl( 
\t W_{t_{k+1}} - \t W_{t_{k}} \bigr) ,\ch \varsigma_{t_{k}} \bigl(
\t W_{t_{k+1}}- \t W_{t_{k}} \bigr) \bigr] =  \bigl( t_{k+1}-t_{k} \bigr) 
D^2 \tilde{H}(\ch \chi_{t_{k}})  \bigl[ \ch \varsigma_{t_{k}} \t G,\ch \sigma_{t_{k}} \t G 
\bigr], 
\end{equation*}
which is enough to prove that
\begin{equation*}
\begin{split}
&\sum_{k=0}^{K-1}
D^2 \tilde{H}(\ch \chi_{t_{k}})  \bigl[\ch  \varsigma_{t_{k}} \bigl( 
\t W_{t_{k+1}} - \t W_{t_{k}} \bigr) , \ch \varsigma_{t_{k}} \bigl(
\t W_{t_{k+1}} - \t W_{t_{k}} \bigr) \bigr]
\rightarrow \int_{0}^T 
D^2 \tilde{H}(\ch \chi_{s})  \bigl[ \ch \varsigma_{s} \t G,\ch \varsigma_{s} \t G 
\bigr] ds 
\end{split}
\end{equation*}
in $\PP^0$ probability as $h$ tends to $0$.

{\small
\bibliographystyle{plain}

}

\end{document}